\documentclass[12pt]{amsart}

\usepackage{amssymb}
\usepackage{amsmath}
\usepackage{amsfonts}
\usepackage[all]{xy} 
\usepackage{mathrsfs,enumitem}
\usepackage[utf8]{inputenc}
\usepackage{graphicx} 
\usepackage{xcolor}
\usepackage{ulem} 
\usepackage{tikz-cd}

\newcommand{\disk}{\ensuremath{\mathbb{D}} } 
\newcommand{\riem}[1][]{\ensuremath{\Sigma_{#1}}} 
\newcommand{\riembase}{\ensuremath{\Sigma'}}

\newcommand{\sphere}{\overline{\mathbb{C}}}

\theoremstyle{plain}
        \newtheorem{theorem}{Theorem}[section]
        \newtheorem{lemma}[theorem]{Lemma}
        \newtheorem{proposition}[theorem]{Proposition}

\theoremstyle{definition}
        \newtheorem{definition}[theorem]{Definition}
        \newtheorem{example}{Example}[section]

\theoremstyle{remark}
    \newtheorem{remark}[theorem]{Remark}

\numberwithin{equation}{section} 
\numberwithin{figure}{section} 

\usepackage{fullpage}

\begin{document}

\author{Eric Schippers}
\author{Wolfgang Staubach}
\address{\newline
       Eric Schippers \newline
       Machray Hall, Dept. of Mathematics,
   University of Manitoba, \newline Winnipeg, MB
   Canada R3T 2N2}
       \email{eric.schippers@umanitoba.ca}

\address{\newline
       Wolfgang Staubach \newline
       Department of  Mathematics, Uppsala University, \newline
       S-751 06 Uppsala, Sweden}
       \email{wulf@math.uu.se}


\title{A unified approach to conformal and modular invariants}

\begin{abstract}  
    In this paper we give a general family of conformal invariants associated to bordered Riemann surfaces endowed with boundary parametrizations, or equivalently compact surfaces endowed with conformal maps. Each invariant is specified by a field of one-forms over a Teichm\"uller space of infinite conformal type. The invariants are positive, and under certain conditions monotonic.
    It is shown that these conformal invariants can be viewed as generalized modular invariants on Teichm\"uller space and as functions on the rigged moduli space of Segal and Vafa. The construction uses an identification of Teichm\"uller space and the rigged moduli space, as well as analytic work of the authors showing that the transfer or ``overfare'' of harmonic functions sharing boundary values on a quasicircle is bounded. 
    
    Demanding invariance under various subgroups of the modular group -- equivalently, under the group of quasisymmetric reparametrizations of a sub-collection of borders -- generates conformal invariants. We show that a wide variety of conformal invariants can be obtained through various choices of the field of one-forms. These include modules of doubly-connected domains, period mappings obtained from harmonic measures, inequalities for higher-order conformal invariants, and the Grunsky inequalities and their recent generalizations to Riemann surfaces.     
\end{abstract}

\dedicatory{Dedicated to the memory of Ian Graham}

\maketitle

\begin{section}{Introduction}
 This paper gives a general framework for conformal invariants associated to conformal embeddings of disks into open Riemann surfaces. That is, invariants which contain information specifying the relative conformal geometry between the ambient surface and the embedded disks. These invariants can depend (for example) on the Taylor coefficients of the conformal embedding map of arbitrary order.\\ 
 
 In this paper we reformulate such relative invariance as a new kind of modular invariance on Teichm\"uller space. We also give explicit constructions of families of examples.\\  

 More precisely, we consider configurations of $k$ conformal maps $f_m:\disk \rightarrow \Omega_m \subset {\riem}_k$ into a Riemann surface ${\riem}_k$.  Here ${\riem}_k$ is a genus $g$ surface with $n-k$ borders homeomorphic to the circle for some $n \geq k$. A conformal invariant is a quantity $I({\riem}_k,\Omega_1,\ldots,\Omega_k,f_1,\ldots,f_k)$ such that 
 \begin{equation} \label{eq:relative_conf_inv_meaning} 
  I(G({\riem}_k),G(\Omega_1),\ldots,G(\Omega_k),G \circ f_1,\ldots, G \circ f_k) = I({\riem}_k,\Omega_1,\ldots,\Omega_k,f_1,\ldots,f_k) 
 \end{equation}
 for any biholomorphism $G$. These are ``relative'' conformal invariants, in the sense that they depend not just on the surface but also how they relate to each other. For a survey of this relative point of view, its roots in geometric function theory and Teichm\"uller theory see \cite{Minda_confinv}, and our survey \cite{Sasha_paper}.

 The family of invariants in this paper unifies many kinds of invariants. They encompass for example hyperbolic reduced module (and Schwarz lemma); the module of an annulus; the Grunsky operator and its generalizations to surfaces of arbitrary finite genus and finite number of boundary curves; the period mapping of harmonic measures in arbitrary genus; and expressions in higher order derivatives of mapping functions. \\

 Relative conformal invariants were investigated by the first author in e.g. \cite{Sch_confinv_Janal}. In particular, an infinite-dimensional family of bounded, monotonic invariants of configurations of nested simply-connected domains was obtained in \cite{Minda_confinv, Schippersinvariants}, which as special cases includes old and new inequalities for bounded univalent maps. Each invariant arises from a choice of quadratic differential. This was an extension of an approach to inequalities of Z. Nehari for maps into a Riemann surface, which derived inequalities from solutions to the Dirichlet problem \cite{Nehari_paper} involving singular holomorphic functions; Nehari's method corresponds to the case that the quadratic differentials in \cite{Minda_confinv, Schippersinvariants} are squares of one-forms \cite{Minda_confinv}.  By tying the quadratic differential to the freely varying outer domain, the inequalities were made manifestly conformally invariant. \\

 The invariants in the present paper are a variation of the invariants in \cite{Minda_confinv, Schippersinvariants}. Here we restrict to the special case that the quadratic differentials are squares of one-forms as in Nehari \cite{Nehari_paper}, but retain the conformally invariant formulation obtained by tying the one-forms to the boundary. However, more importantly, we place them in the context of Teichm\"uller theory and modular invariants.  This requires extending the invariants to domains bounded by quasicircles, which introduces analytic issues arising from treating boundary value problems on quasicircles. These issues were resolved by the authors in a series of papers e.g \cite{EMS_survey,Schippers_Staubach_scattering_I,Schippers_Staubach_scattering_II}. The boundary value problems are associated to a kind of geometric scattering theory developed in \cite{Schippers_Staubach_scattering_III,Schippers_Staubach_scattering_IV}. 
 It also required a geometric and analytic result of D. Radnell and E. Schippers \cite{RadnellSchippers_monster} which identifies the rigged moduli space of G. Segal and C. Vafa with Teichm\"uller space. For the role of the Segal--Vafa moduli space in two-dimensional conformal field theory see Y.-Z. Huang \cite{Huang_book,Huang_open_problems}, and for a general discussion of the consequences of its relation to Teichm\"uller space see \cite{RSS_Lepowsky_survey,ThurstonEW}.\\

 There are two main results. The {\bf first result} is that this kind of relative conformal invariance is equivalent to  functions on Teichm\"uller space which are invariant under certain subgroups of the modular group. Namely, given a Riemann surface $\riem$ with $n \geq 0$ borders homeomorphic to the circle, we consider those elements of the modular group which are the identity on $k$ borders, but not necessarily on the remaining ``free'' boundaries.  Invariants under this subgroup of the modular group are in one-to-one correspondence with conformal invariants of the relative type  (\ref{eq:relative_conf_inv_meaning}) above.\\ 
 
 This correspondence is formulated with the help of a result of Radnell and the first author {\cite{RadnellSchippers_monster}} which says that the Teichm\"uller space and the Segal-Konstevich-Vafa-Friedan-Shenker moduli spaces are identified up to a discontinuous group action. Through this identification we are able to identify the modular groups with groups of quasisymmetric reparametrizations of the boundary. The quasisymmetries are a natural completion of the group of diffeomorphisms of the circle. This allows us here to greatly generalize the classical identification of modular invariants with conformal invariants depending on a single surface.

 The largest modular group contains one copy of the group of quasisymmetries of the circle. This can be thought of as a completion of the diffeomorphism group of the circle (where the completion is required to make the geometric connection to Teichm\"uller theory \cite{RadnellSchippers_monster,Sasha_paper}.) The quasisymmetries act by reparametrizations of the boundary curves. By demanding invariance under reparametrization of a particular boundary, one makes that boundary ``free''. If a boundary is free, then the invariant does not depend on how that boundary sits in the outer surface.

 The {\bf second result} is that we obtain an explicit family of monotonic and positive invariants. These are obtained as follows. For a fixed type of Riemann surface (in this paper, $g$ handles with $n$ boundary curves), consider the bundle of Riemann surfaces over Teichm\"uller space, where each fiber is a representative surface in the Teichm\"uller equivalence class. This is the Teichm\"uller curve. Consider the vector bundle over Teichm\"uller space whose fibres are the Hilbert space of $L^2$ holomorphic one-forms on each surface. From each section of this vector bundle satisfying certain conditions which will be made explicit, we obtain a modular invariant.

 The construction of this family requires that a process we call overfare {(introduced in  \cite{Schippers_Staubach_scattering_I,Schippers_Staubach_scattering_II}}) is well-defined and bounded. Overfare was studied by the authors in association with a scattering theory for one-forms, and is defined as follows. Given a Riemann surface split by a family of quasicircles, and an $L^2$ harmonic one-form on one side of a quasicircle, there is an $L^2$ one-form on the other side with the same boundary values and specified cohomology, and the map from one to the other is bounded. This was proven by the authors in \cite{Schippers_Staubach_scattering_I,Schippers_Staubach_scattering_II}. 
 \\

 We would like to emphasize that the aforementioned analytic and geometric results, such as boundedness of the overfare operator and the Teichm\"uller space/rigged moduli space correspondence, pave the way for the proofs of the main theorems in this paper. \\

 In summary, this paper marks a significant breakthrough in addressing a problem posed in \cite[Problem 4.1]{Sasha_paper}, which asks how to extend the invariants of Theorem \ref{th:its_Nehari} (more precisely, their generalizations to arbitrary quadratic differentials, obtained in \cite{Schippersinvariants}) to the moduli spaces of configurations bounded by quasicircles. This problem has both a geometric aspect and an analytic one. Geometrically, one needs to make sense of the definition on Teichm\"uller space or the rigged moduli space, and give the geometric meaning of the objects in the definition. Analytically, one must deal with the irregularity of the curves, which was accomplished in \cite{Schippers_Staubach_scattering_I,Schippers_Staubach_scattering_II}.

The paper is organized as follows. In Sections 2.1-2.5 we set up the stage for the conformal and modular invariants discussed in the subsequent sections by give an overview of spaces of harmonic and holomorphic one-forms, Teichm\"uller spaces, Teichm\"uller  curves, Riemann surfaces with caps and rigged moduli spaces and modular groups. In Section \ref{se:modular_invariants_are_conformal_invariants}
 we prove a correspondence between functions on Teichm\"uller space which are invariant under subgroups of the modular groups and conformal invariants on the moduli space of surfaces and conformal maps with free boundaries. At the root of this is the correspondence between Teichm\"uller space and the rigged moduli space. This shows how modular invariance under the groups reparametrizations of particular boundary curves renders these boundaries ``free''. 
 In Section \ref{se:invariants_from_sections} we describe how to get specific conformal invariants from norms of sections of the vertical cotangent bundle, and demonstrate the equivalence of relative invariance and modular invariance. In Section \ref{se:sections_from_data} we show how to obtain sections from particular data, and a general construction of sections leading to positive modular invariants, which are also under certain conditions monotonic.  We see that if a particular boundary is a trajectory of the one-forms then the quantity is invariant under the group of reparametrizations of that boundary.  In Section \ref{se:equality} we investigate the case that the lower bound of zero is attained. Section \ref{se:Schiffer_operators} defines certain integral operators of Schiffer, and in Section \ref{se:Faber-Tietz} we define certain forms associated to conformal maps into a compact surface called Faber-Tietz forms, obtained by the first author with Shirazi \cite{SchiShi_Faber}. The Schiffer operators and the Faber-Tietz forms are instrumental in the construction of many of the examples.  Finally, Section \ref{se:examples} is dedicated to specific examples of conformal invariants which can be described using the method of this paper. \\

 {\bf{Acknowledgements.}} The first author was partially supported by the National Sciences and Engineering Research Council of Canada. The second author is grateful to Andreas Strömbergsson for partial financial support through a grant from Knut and Alice Wallenberg Foundation. We are also grateful to Eveliina Peltola for discussions regarding an early draft of this paper.
\end{section}

\begin{section}{Moduli spaces and modular groups}
 \label{se:definitions_notation_etc}
 \begin{subsection}{Preliminaries}

 We define some function spaces on Riemann surfaces. In this section $R$ denotes a generic Riemann surface, which we take to be either a compact Riemann surface possibly with punctures, or a surface of genus $g$ with $n$ borders homeomorphic to $\mathbb{S}^1$, again possibly with punctures. Here and in what follows, when we say a surface is compact, we mean that it has no boundary, except possibly the specified punctures. 

 A {\it quasiconformal map} $F:R \rightarrow R_1$ between Riemann suraces $R$ and $R_1$ is a deformation of the complex structure. The precise definition is as follows.  We say that a function $f$ on an open subset $U$ of the plane $\mathbb{C}$ is absolutely continuous on lines if on every rectangle $[a,b] \times [c,d] \subset U$, $f(x,y)$ is absolutely continuous in $y$ for almost every fixed $x$ in $[a,b]$ and is absolutely continuous in $x$ for almost every $y \in [c,d]$.  We say that a map $F:R \rightarrow R_1$ between Riemann surfaes is absolutely continuous on lines if it is absolutely continuous on lines in local coordinates for every choice of coordinates on $R$ and $R_1$. In particular, the partial derivatives exist almost everywhere and we can make the following definition.
 \begin{definition}\label{defn:QC} A quasiconformal map is 
 an orientation-preserving homeomorphism which is absolutely continuous on lines, and such that 
 \[ \left\|  \frac{\overline{\partial} F}{\partial F}  \right\|_{L^\infty} <1.  \]
 We shall denote the set of quasiconformal self-maps of $R$ by $\mathrm{QC}(R)$.
 \end{definition}
 
 A quasiconformal map has a unique continuous (indeed homeomorphic) extension to the punctures and borders. 

 Let $\mathbb{S}^1$ denote the circle, which we identify with the boundary of the unit disk $\disk = \{ z\,:\, |z| <1 \}$. 
\begin{definition}\label{def: QS}
A {\it quasisymmetry} of the circle $\mathbb{S}^1$ is an orientation-preserving homeomorphism which is the boundary values of a quasiconformal map of $\disk$.\\
  \end{definition}

  Alternatively a homeomorphism $h:\mathbb{S}^1 \to \mathbb{S}^1$ is a quasisymmetry if there is a constant $k>0$, such that for every $\alpha$, and every $\beta$ not equal to a multiple of $2\pi$, the inequality
 \[  \frac{1}{k} \leq \left| \frac{h(e^{i(\alpha+\beta)})-h(e^{i\alpha})}{h(e^{i\alpha})-h(e^{i(\alpha-\beta)})} \right|
    \leq k \]
 holds. The two definitions given above are equivalent \cite{Lehto_book}. The orientation-preserving (orientation-reversing) quasisymmetries will be denoted by $\mathrm{QS}_+(\mathbb{S}^1)$ ($\mathrm{QS}_-(\mathbb{S}^1)$).

 Assume that $\Gamma$ is a border of a Riemann surface $R$ homeomorphic to $\mathbb{S}^1$. Any such border has a collar neighbourhood $U$ and a biholomorphism $\psi:U \rightarrow \mathbb{A}_{r,1}$ where $\mathbb{A}_{r,1} = \{ z : r<|z|<1 \}$ for some $r$ such that $0<r<1$; this biholomorphism extends by Carath\'eodory's theorem to a homeomorphism from $\Gamma$ to $\mathbb{S}^1$.  
 We say that $\phi:\mathbb{S}^1 \rightarrow \Gamma$ is an  an orientation-preserving/orientation-reversing quasisymmetry if the map $\psi \circ \phi$ is an orientation-preserving/orientation-reversing quasisymmetry of $\mathbb{S}^1$. 

 On any Riemann surface,  define the dual of the almost 
 complex structure,  $\ast$ in local coordinates $z=x+iy$,  by 
 \[  \ast (a\, dx + b \, dy) = a \,dy - b \,dx. \]
This is independent of the choice of coordinates.
 Let $U$ be an open subset of a Riemann surface $R$. 
 Denote by $L^2(U)$ the set of one-forms $\omega$ on an open set $U$ which satisfy
\[   \iint_U \omega \wedge \ast \overline{\omega} < \infty  \]
(observe that the integrand is positive at every point, as can be seen by writing the expression in local coordinates).  
This is a Hilbert space with respect to the inner product
\begin{equation} \label{eq:form_inner_product}
 (\omega_1,\omega_2) =  \iint_U \omega_1 \wedge \ast \overline{\omega_2}.
\end{equation}

\begin{definition}\label{defn:bergman spaces}
The  {\it Bergman space} of holomorphic one forms is 
\begin{equation}
    \mathcal{A}(U) = \{ \alpha \in L^2(U) \,:\, \alpha \ \text{holomorphic} \}.
\end{equation} 
 The anti-holomorphic Bergman space is denoted $\overline{\mathcal{A}(U)}$.   We will also denote 
\begin{equation}
    \mathcal{A}_{\mathrm{harm}}(U) =\{ \alpha \in L^2(U) \,:\, \alpha \ \text{harmonic} \}.
\end{equation}
\end{definition}
We also consider one-forms which have zero boundary periods.
\begin{definition} \label{de:semi_exact}
 {{Let $\riem$ be a bordered surface of type $(g,n)$.}} We say that an $L^2$ one-form $\alpha \in \mathcal{A}_{\mathrm{harm}}(\riem)$ is {\it semi-exact} if for any simple closed curve $\gamma$ isotopic to a boundary curve $\partial_k \riem$, 
 \[ \int_{\gamma}  \alpha =0. \]
\end{definition}
 \end{subsection}
\begin{subsection}{Teichm\"uller space and Teichm\"uller curve}
 In this section we define the moduli spaces. $R$ continues to denote a generic Riemann surface; either compact or with $n$ borders homeomorphic to $\mathbb{S}^1$, possibly with punctures in either case. Teichm\"uller space is defined for a greater variety of surfaces but this paper concerns only these cases. 
 
 The Teichm\"uller space of a surface $R$ is defined as follows. We say that a quasiconformal map $G:R \rightarrow R$ is {\it homotopic to the identity rel boundary} if it is homotopic to the identity via a homotopy which is the identity on the punctures and border.  
\begin{definition} \label{de:Teich_space_and_eq_rel} Let $R$ be a Riemann surface with $g$ handles, $n$ borders homeomorphic to $\mathbb{S}^1$, and possibly a finite number of punctures. Given two quasiconformal deformations $F_1: R\to R_1$ and $F_2: R\to R_2$  we say that they are {\it Teichm\"uller equivalent} if there is a conformal map $\sigma:R_1 \rightarrow R_2$ such that $F_2^{-1} \circ \sigma \circ F_1$ is homotopic to the identity rel boundary. The equivalence class of $F_1$ will be denoted by $[R,F_1,R_1].$

 The {\it Teichm\"uller space} of $R$ is then defined as 
 \[  \mathscr{T}(R) = \{ (R,F_1,R_1) \}/\sim, \]
 where $\sim$ is the Teichm\"uller equivalence relation.
 \end{definition}

 We remark that $F:R \rightarrow R_1$ is often referred to as a {\it marking map} (and $R_1$ is referred to as {\it a marked surface}), especially in the case that $R_1$ is compact with punctures. In that special case, this map  can be thought of as specifying a basis of generators for the homology group of $R_1$.\\ 

 The {\it Teichm\"uller curve} $\mathcal{T}(R)$ is a holomorphic fiber space over $\mathscr{T}(R)$,  with a holomorphic fibration (projection) map $\mathfrak{B}:\mathcal{T}(R) \rightarrow \mathscr{T}(R)$ with the property that the fiber $\mathfrak{B}^{-1}(q)$ over $q \in \mathscr{T}(R)$ is a specific Riemann surface $R_1$ in the equivalence class of marked surfaces determined by $q$, in which case $q=[R,F_1,R_1]$ for some $F_1:R \rightarrow R_1$. Elements of $\mathcal{T}(R)$ can be represented by pairs $(q,z)$ where $q \in \mathscr{T}(R)$ and $z \in \mathfrak{B}^{-1}(q)$. The Teichm\"uller curve is a Banach manifold and $\mathfrak{B}$ is a holomorphic fiber map, see \cite{Lehto_book, Nag_book}.\\  

 The vertical tangent bundle $T_{\textnormal {vert}} \mathcal{T}(R)$ is the space of tangent vectors tangent to the fibers, and the vertical cotangent bundle $T_{\textnormal {vert}}^* \mathcal{T}(R)$ is its dual space. These can be thought of as elements of the tangent and cotangent bundles $T {\riem}^q$ and $T^* {\riem}^q$ respectively, where ${\riem}^q = \mathfrak{B}^{-1}(q)$ is the surface corresponding to any fixed fiber over a specific $q \in \mathscr{T}(R)$.  These refer to the complex pure tangent spaces and cotangent spaces; that is, vectors and one-forms given in a local coordinate $z$ by $h(z) \partial/\partial z$ and $g(z) dz$ respectively.   \\

 Let $A^2(T_{\textnormal {vert}}^* \mathcal{T}(R))$ denote the sections of the vertical cotangent bundle 
 \[ \hat{s}:\mathcal{T}(R) \rightarrow T_{\textnormal {vert}}^* \mathcal{T}(R)  \]
 which are holomorphic and such that the restriction to any vertical cotangent space is in $L^2$. More precisely, if $s(q)$ is defined by $s(q)(z) = \hat{s}(q,z)$ then $s(q)$ is an element of the Bergman space of $\mathfrak{B}^{-1}(q)$.
 
 \begin{definition} \label{de:field_of_oneforms} Given a Riemann surface $R$ and a section
  $\hat{s} \in A^2(T_{\textnormal {vert}}^* \mathcal{T}(R))$ we call the family $s$ of one forms 
  \[  s(q) \in \mathcal{A}(\mathfrak{B}^{-1}(q))  \ \ \ \ q \in \mathscr{T}(R)  \] defined by 
    \begin{equation} \label{eq:field_of_oneforms_defn}
   s(q)(z) = \hat{s}(q,z)  
 \end{equation}
 a {\it field of one-forms} over Teichm\"uller space $\mathscr{T}(R)$.  
 \end{definition}
 \begin{remark} 
  We can also describe a field of one forms in the following way. Consider the  bundle $\pi:E \rightarrow \mathscr{T}(R)$ whose fibers are $\pi^{-1}(q) = \mathcal{A}(\mathfrak{B}^{-1}(q))$. Then
  \begin{align*}
   s: \mathscr{T}(R) & \rightarrow E\\
   q \mapsto s(q)
  \end{align*}
  is a section of $E$.  
 \end{remark}
 
 Summarizing the notation that will be used in this connection:
 \begin{itemize}
  \item $\mathscr{T}(R)$ is the Teichm\"uller space of $R$; 
  \item $\mathcal{T}(R)$ is the Teichm\"uller curve of $R$, with fibre map $\mathfrak{B}:\mathcal{T}(R) \rightarrow \mathscr{T}(R)$; 
  \item $T_{\textnormal {vert}}\mathcal{T}(R)$ is the vertical tangent bundle;
  \item $T_{\textnormal {vert}}^* \mathcal{T}(R)$ is the vertical cotangent bundle;
  \item $A^2(T_{\textnormal {vert}}^* \mathcal{T}(R))$ are the holomorphic sections of $T_{\textnormal {vert}}^* \mathcal{T}(R)$ whose restrictions to each fixed vertical fiber are $L^2$.
 \end{itemize}  

\end{subsection}

\begin{subsection}{Riemann surfaces with caps and conformal maps}\label{Riemann surfaces with caps}
 We consider configurations of Riemann surfaces with ``caps''.  These are as follows.
 Let $\mathscr{R}$ be a compact surface with punctures $p_1,\ldots,p_n$. Let $\Omega_1,\ldots,\Omega_n$ be simply connected quasidisks in $\mathscr{R}$ containing $p_1,\ldots,p_n$ respectively, whose closures do not intersect. Let $\Omega = \Omega_1 \cup \cdots \cup \Omega_n$ and let $\riem$ be the complement of the closure of $\Omega$ in $\mathscr{R}$.  
 The letters $\mathscr{R}$, $\Omega$, and $\riem$ will have this meaning consistently in the remainder of the paper. 

 We may obtain capped surfaces from a bordered surface $\riem$ by sewing disks on via boundary parametrizations, as will be described in the next section.  In the reverse direction, we can start with $\mathscr{R}$ and mappings $f=(f_1,\ldots,f_n)$ and remove the closures of the images $\Omega_m=f_m(\disk)$ to obtain the bordered surface $\riem$.  
 
 When removing some caps but not others, we use the following notation.  For fixed $k=1,\ldots,n-1$ we define the surface ${\riem}_k$ to be 
 \[ {\riem}_k := \mathscr{R} \backslash \mathrm{cl} \, \left[ \Omega_{k+1} \cup \cdots \cup \Omega_n\right]. \]
 We also set ${\riem}_0 = \riem$ and ${\riem}_n=\mathscr{R}$. Thus we have the chain of inclusions 
 \[  \riem={\riem}_0 \subset {\riem}_1 \subset \cdots \subset {\riem}_n = \mathscr{R}.  \]
 Of course, one could remove the caps in any order, but we fix an ordering to simplify notation. All of the results in this paper hold for an arbitrary ordering, since the choice of ordering has no bearing on the proofs.\\  

  {\bf The following notational conventions will be in force throughout the paper.} 
 \begin{itemize}
     \item $\mathscr{R}$ will be a compact Riemann surface with punctures $p_1,\ldots,p_n$;
     \item $\Omega= \Omega_1 \cup \cdots \cup \Omega_n$ will be the caps, as described above;
     \item $\riem$ will be the complement of the closures caps in $\mathscr{R}$;
     \item ${\riem}_0 \subset {\riem}_1 \subset \cdots \subset {\riem}_{n-1} \subset {\riem}_n$ will be the chain of Riemann surfaces obtained by adding one cap at a time as above.
 \end{itemize}
 We will use primes when more than one such configuration 
 is involved, e.g. $\mathscr{R}'$, $p_k'$, $\Omega_k'$, etc. 
\end{subsection}
\begin{subsection}{The rigged moduli space}\label{rigged moduli space}
 We will also consider the Friedan-Shenker-Segal-Vafa moduli space, which we call the rigged moduli space. This appears in conformal field theory; here, we have chosen the analytic conditions so that a connection between the moduli space and Teichm\"uller space can be drawn, as in \cite{RadnellSchippers_monster}. For a discussion of the analytic choice and its ramifications see \cite{RSS_Lepowsky_survey}. 
 
 There are two equivalent models of the rigged moduli space, one involving Riemann surfaces with boundary and parametrizations of the boundary curves, and the other involving compact Riemann surfaces with punctures and conformal maps into a neighbourhood of the punctures. We refer to these two models as the {\it border} and {\it puncture model} respectively.\\ 
 
 We consider pairs $(\riem,\phi)$ where $\riem$ is a Riemann surface with $n$ borders homeomorphic to $\mathbb{S}^1$, and $\phi=(\phi_1,\ldots,\phi_n)$ is an $n$-tuple of orientation-reversing quasisymmetric parametrizations $\phi_k:\mathbb{S}^1 \rightarrow \riem$.   Call such a pair a {\it rigged Riemann surface}.
 \begin{definition}[Rigged moduli space, border model]
 We say that two pairs of rigged surfaces are equivalent $(\riem,\phi) \sim ({\riem}',{\phi}')$ if there is a biholomorphism $\sigma:\riem \rightarrow {\riem}'$ such that $$\phi_k'=\sigma \circ \phi_k$$ for $k=1,\ldots,n$. The {\it rigged moduli space} is 
 \[  \widetilde{\mathcal{M}}(g,n) = \{ (\riem,\phi) \}/\sim   \]
 and we denote its elements by $[\riem,\phi]$. 

 \end{definition}
 The puncture model is set up as follows.
 Fix a compact surface $\mathscr{R}$ of genus $g$ with $n$ punctures $p_1,\ldots,p_n$, which are given a definite ordering.\\
 
 A {\it rigging} is a collection of conformal maps $f=(f_1,\ldots,f_n)$ $f_k:\disk \rightarrow \mathscr{R}$ such that $f_k(0)=p_k$ for all $k$, and such that the closures of the images of $f_k$ do not intersect. We also assume that these maps are quasiconformally extendible to an open neighbourhood of the closure of $\disk$. 

 \begin{definition}[Rigged moduli space, puncture model]
 We say that a pair of punctured surfaces with riggings $(\mathscr{R},f)$ and $(\mathscr{S},g)$ are equivalent if there is a conformal map $\sigma:\mathscr{R} \rightarrow \mathscr{S}$ preserving the ordering of the punctures such that 
 \[   g_k= \sigma \circ f_k       \] 
 for $k=1,\ldots,n$. The rigged moduli space in this model is 
 \[  \widetilde{\mathcal{M}}_P(g,n) = \{ (\mathscr{R},f) \}/\sim.   \]
 \end{definition}
 
 Equivalence classes will be denoted by $[\mathscr{R},\Omega,f]$.  Although the notation is redundant since $\Omega$ can be recovered from $f$,  
  it will be convenient to include the caps $\Omega$ in the notation.\\ 
  
 These two moduli spaces are naturally in one-to-one correspondence. We give an outline of the correspondence  in the present formulation; for proofs see \cite{RadnellSchippers_monster}. 
 
 Fix a base surface $\riem$ with $n$ borders and $g$ handles as above, and a rigging $\phi=(\phi_1,\ldots,\phi_n)$. 
 We can sew copies of the disk  by identifying points $p \in \mathbb{S}^1$ with $\phi_k(p)$ in $\partial_k \riem$ to obtain a topological surface.  This new surface has a unique complex structure which agrees with that of $\riem$ and the copies of the disk \cite[Theorems 3.2 and 3.3]{RadnellSchippers_monster}.
 Remove the point $0$ in each copy of the disk to obtain a punctured Riemann surface $\mathscr{R}$.  
 The resulting seams are quasicircles, and each $\phi_k$ has a holomorphic and bijective extension to a map $f_k$ from $\disk$ to $\mathscr{R}$, and we label the image by $\Omega_k$ and set $p_k = f_k(0)$.  This map is quasiconformally extendible. 

 This induces a well-defined bijection 
 \begin{align*}
     \Phi:\widetilde{\mathcal{M}}_B(g,n) & \rightarrow \widetilde{\mathcal{M}}_P(g,n) \\
     [\riem,\phi] & \rightarrow [\mathscr{R},\Omega,f]
 \end{align*}
 where $\mathscr{R}$ and $f$ are obtained from $\riem,\phi$ as above (see Figure \ref{rigged_to_rigged}).
 \begin{figure}[h] 
  \caption{Correspondence of border and puncture models pf rigged moduli space}\label{rigged_to_rigged} 
\includegraphics[width=12cm]{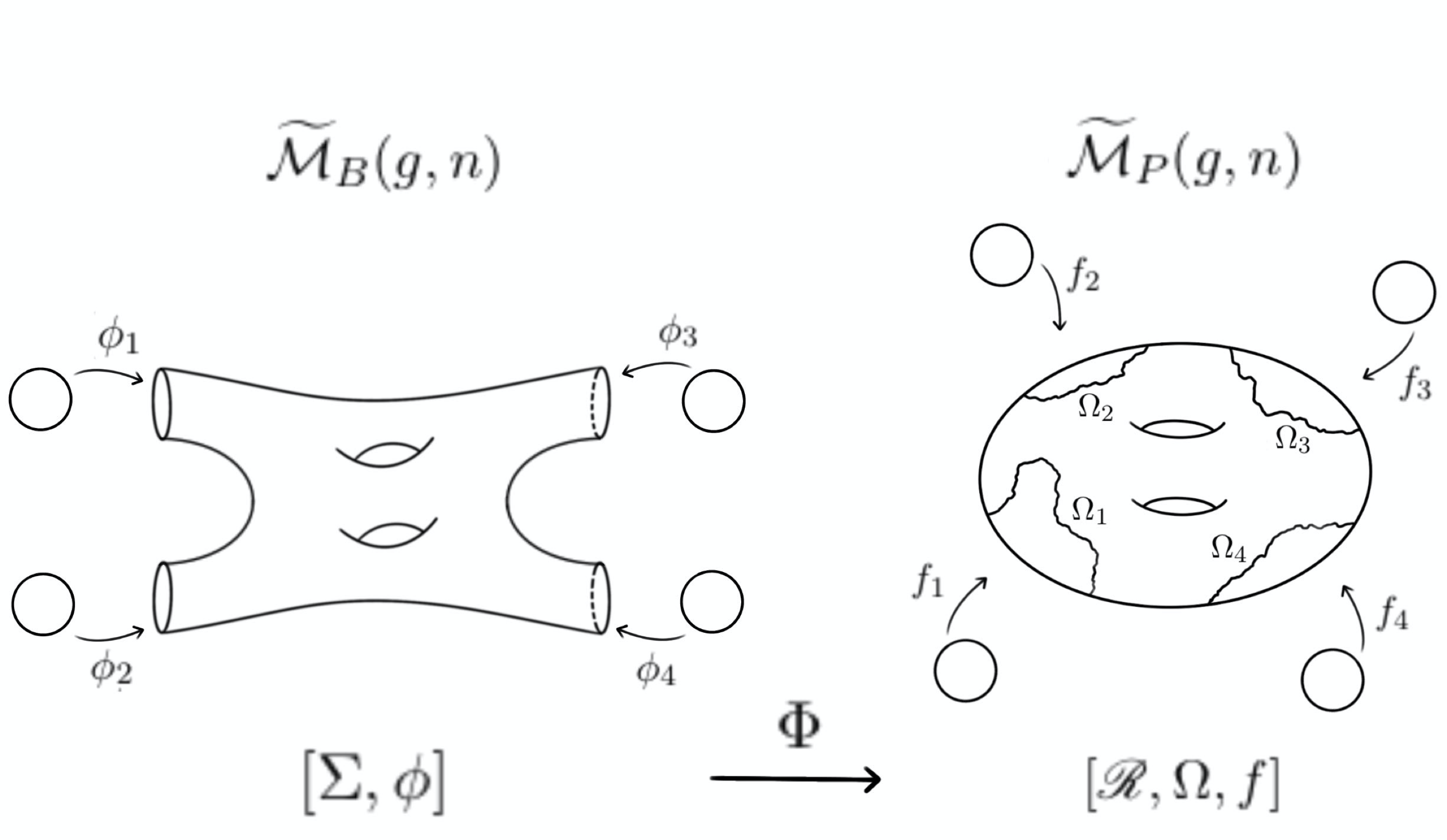}

\end{figure}
 
 \begin{remark}
     It may initially seem strange not to picture $\Omega_k$ as a disk; to understand this, we can look at the process in the Riemann sphere. Let $\riem= \disk_-=\{z \,: \, |z|>1 \} \cup \{\infty \}$. Given a quasisymmetry $\phi:\mathbb{S}^1 \rightarrow \mathbb{S}^1$, we sew on $\disk$ using the parametrization $\phi$. The resulting Riemann surface $\disk \# \disk_-$ is biholomorphic to the Riemann sphere $\sphere$; after uniformizing via a biholomorphism $g: \disk \# \disk_- \rightarrow \sphere$, the resulting map $\left. g \right|_{\disk}: \disk \rightarrow \sphere$ is a biholomorphism onto a quasidisk $\Omega$. That is, in the uniformized picture, the disk becomes a domain $\Omega$, whose boundary might be quite irregular.
     
     If one views $\disk \# \disk_-$ abstractly, then $g$ is the identity on the seam when the seam is viewed as $\partial \disk$, but is equal to $\phi$ when the seam is viewed as $\partial \disk_-$.  
     This is the well-known process of conformal welding \cite{Lehto_book,Nag_book}. Here we extend this to Riemann surfaces as in \cite{RadnellSchippers_monster}. 
     
     In our abstract formulation, the uniformizing map $g$ does not appear explicitly, but is implicit in the statement ``there exists a unique complex structure compatible with ...''.

     In the case that $\phi_k$ is a diffeomorphism, the seam is a smooth curve. 
 \end{remark} 

\end{subsection} 
\begin{subsection}{Modular groups}

 In the following, $R$ is a Riemann surface with $n$ borders homeomorphic to $\mathbb{S}^1$, $g$ handles, and possibly a finite number of punctures.   
 Given a Riemann surface $R$, let $\mathrm{PQC}(R)$ denote the set of quasiconformal self-maps of $R$ which preserve each puncture and each border set-wise.  Note that $\mathrm{PQC}(R)$ is a group under composition. For convenience as above, we give the borders an ordering $\partial_1 R,\ldots,\partial_n R$. Let $\mathrm{PQC}_I(R)$ denote the subset of those self-maps which are homotopic to the identity via a homotopy which is the identity on the boundary. This is a normal subgroup of $\mathrm{PQC}(R)$. We say that $\rho_1,\rho_2 \in \mathrm{PQC}(R)$ are equivalent if $\rho_1 \circ \rho_2^{-1} \in \mathrm{PQC}_I(R)$. 
 The pure modular group of a Riemann surface $R$ is defined to be 
 \[  \mathrm{PMod}(R) = \mathrm{PQC}(R)/\mathrm{PQC}_I(R).  \]
 Since we will only consider pure modular groups, we will drop the adjective "pure". The modular group acts on the Teichm\"uller space via 
 \begin{equation}\label{defn:rho action}
      [\rho] [\riem,F,{\riem}_1] = [\riem,F \circ \rho^{-1},{\riem}_1].
 \end{equation}   
 The quotient of Teichm\"uller space by the action of the {\it full} modular group (that is, not just pure) is the Riemann moduli space.     

 Now we let $\riem$ be a Riemann surface with $g$ handles and $n$ borders homeomorphic to $\mathbb{S}^1$, with no punctures. 
 We consider a number of subgroups of $\mathrm{PMod}(\riem)$:
 \begin{itemize}
  \item $\mathrm{PModI}(\riem)$ is the subgroup generated by quasiconformal self-maps which are the identity on every border; 
  \item $\mathrm{PModI}_k(\riem)$ is the subgroup generated by quasiconformal self-maps which are the identity on the borders $\partial_1 \riem,\ldots,\partial_k \riem$; 
  \item $\mathrm{DI}(\riem)$ is the subgroup generated by Dehn twists around the handles;
  \item $\mathrm{DB}(\riem)$ is the subgroup generated by Dehn twists around the boundary curves;
  \item $\mathrm{DB}(\partial_k \riem)$ is the subgroup generated by Dehn twists around the specific boundary curve $\partial_k \riem$;
 \end{itemize}

We then have the chain of inclusions 
\[  \mathrm{PModI}_n(\riem) \subset \mathrm{PModI}_{n-1}(\riem) \subset \cdots \subset \mathrm{PModI}_1(\riem) \subset \mathrm{PModI}_0(\riem)   \]
 where the end cases $k=0,n$ are given by 
 \[  \mathrm{PModI}_0(\riem) = \mathrm{PMod}(\riem) \]
 and 
 \[   \mathrm{PModI}_n(\riem)=\mathrm{PModI}(\riem). \]

 We get an identification of $\mathrm{QS}(\mathbb{S}^1)$ with a subgroup of $\mathrm{PMod}(\riem)$ via the following prescription.  Fix an $n$-tuple  $\tau = (\tau_1,\ldots,\tau_n)$ of orientation-reversing quasisymmetries $\tau_k:\mathbb{S}^1 \rightarrow \partial_k \riem$. 
 Fix a particular $k$ and let $U_k$ be collars of $\partial_k \riem, \ldots, \partial_n \riem$ bordered by analytic curves. For each element $\psi \in \mathrm{QS}(\mathbb{S}^1)$ we get an induced quasisymmetry of $\partial_k \riem$ via 
 \begin{equation} \label{eq:qs_modular_identification}
   \psi \mapsto \tau_k \circ \psi \circ \tau_k^{-1}.  
 \end{equation}

 Let 
 \[  \mathrm{QS}(\partial_k \riem)    
 = \{ \tau_k \circ \psi \circ \tau_k^{-1} \,:\, \psi \in \mathrm{QS}(\mathbb{S}^1)  \}. \]
 We can identify this with the subgroup of $\mathrm{PMod}(\riem)$ consisting of $[\rho] \in \mathrm{PMod}(\riem)$ such that $\rho$ is the identity on all boundaries except $\partial_k \riem$, and such that $\rho$ is homotopic to the identity (where the homotopy need not be the identity on the boundary).  We will use the same notation for this subgroup, since it will only be used heuristically in this paper.  Intuitively one can see that 
  \[   \mathrm{PModI_{k}}(\riem)/ \mathrm{PModI_{k-1}}(\riem) \cong   \mathrm{QS}(\partial_k \riem)  \]
 and in fact this can be shown using extension theorems for quasisymmetries. 
 
\end{subsection}

\begin{subsection}{Relation between the Teichm\"uller space and the rigged moduli space}
 The rigged moduli space of Riemann surfaces of genus $g$ with $n$ borders homeomorphic to $\mathbb{S}^1$ is the quotient of the Teichm\"uller space of surfaces of the same type by the modular group $\mathrm{PModI}(\riem)$ in a natural way \cite{RadnellSchippers_monster}. We outline this now.
 
 Fix a bordered Riemann surface $\riembase$ with $g$ handles and $n$ borders homeomorphic to $\mathbb{S}^1$, which will play the role of a base surface in the Teichm\"uller space. We also fix orientation-reversing quasisymmetric parametrizations 
 $\tau_k:\mathbb{S}^1 \rightarrow \partial_k \riembase$, $k=1,\ldots,n$. Denote the $n$-tuple by $\tau$, which we call a ``base'' rigging. The covering of the rigged moduli space by the Teichm\"uller space is given by the map
 \begin{align*}
  \mathcal{F}_{\tau}:\mathscr{T}(\riembase) & \rightarrow \widetilde{\mathcal{M}}_B(g,n) \\
  [\riembase,F,{\riem}] & \rightarrow [{\riem},F \circ \tau]
 \end{align*}
 (see Figure \ref{fig:Teich_to_rigged}). 
 \begin{figure}[h] 
  \caption{covering of rigged moduli space by $\mathscr{T}(\riem')$} \label{fig:Teich_to_rigged} 
\includegraphics[width=12cm]{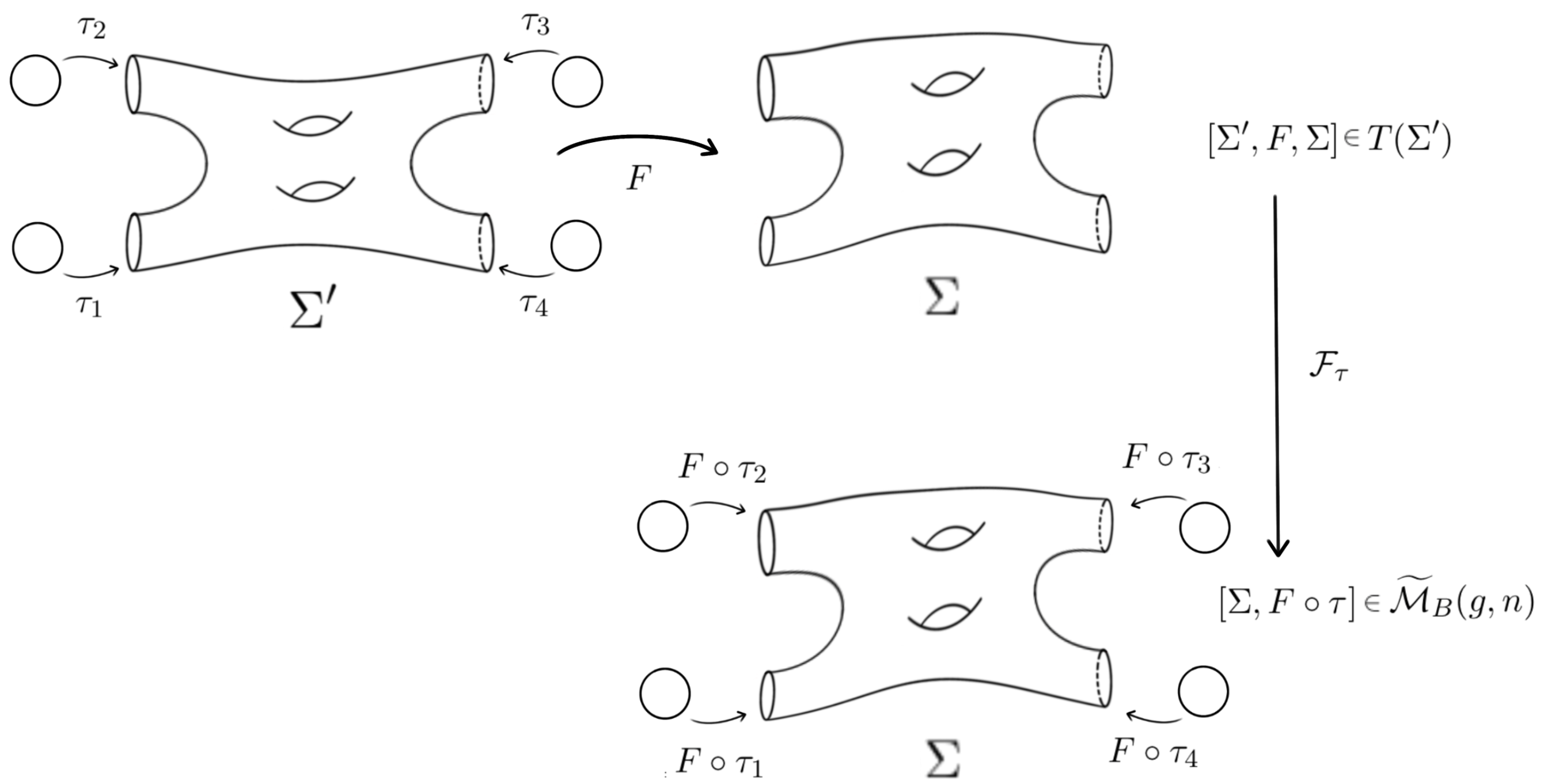}
\end{figure} 
 
 Observe that the composition $F \circ \tau$ is defined only on $\mathbb{S}^1$.

 \begin{remark}
 A change in $\tau$ is associated to a change of base point in Teichm\"uller space, but we will not pursue this point. 
 \end{remark}

 We have the following result \cite[Corollary 5.2 and 5.3]{RadnellSchippers_monster}.
 \begin{theorem} \label{th:F_tau_defined_up_to_PModI}
  $\mathcal{F}_\tau(p)= \mathcal{F}_\tau(q)$ if and only if $p= [\rho]q$ for some $[\rho] \in \mathrm{PModI}(\riembase)$, where the action of $[\rho]$ is given by \eqref{defn:rho action}. Thus
  \[   \mathscr{T}(\riembase)/\mathrm{PModI}(\riembase)  \cong  \widetilde{\mathcal{M}}_B(g,n).    \]
 \end{theorem}
 
 \begin{remark}
 This fact implies that the rigged moduli space has a complex Banach manifold structure. This does not play a role in the paper, though it is worth noting that all the group actions are biholomorphisms \cite{RadnellSchippers_monster}.
 \end{remark}

 The full modular group $\mathrm{PMod}(\riembase)$ induces an action by quasisymmetries on the boundary.  However 
 the identification of $\mathrm{QS}(\mathbb{S}^1)$ with the action on the boundary depends on the choice of base boundary parametrization $\tau$; in many ways, this is analogous to a choice of chart in a manifold. 
\end{subsection}
\begin{subsection}{Notation for Teichm\"uller curve and capped surface}  In this section we outline some notation which will simplify the presentation.

 Assume that the base surface $\riembase$ is a surface of genus $g$ with $n$ borders homeomorphic to $\mathbb{S}^1$. Fix a base parametrization $\tau$, and sew on caps to obtain $\mathscr{R}$ and $\Omega$ as in Section \ref{rigged moduli space}.  

 Let $\mathcal{T}(\riembase)$ denote the universal Teichm\"uller curve with the associated fibration 
 \[  \mathfrak{B}:\mathcal{T}(\riembase)  \rightarrow \mathscr{T}(\riembase) . \] 
 
 Fix $q \in \mathscr{T}(\riembase)$, let $\riem^q = \mathfrak{B}^{-1}(q)$, and let $(\riembase,F^q,{\riem}^q)$ be a representative of $q$. Sewing on caps via $F^q \circ \tau$ we obtain a Riemann surface $\mathscr{R}^q$ with caps $\Omega^q$ and corresponding riggings $f^q$.
 
 Summarizing the notation, we have 
 \begin{align} \label{eq:fibres_description_teich_curve}
  q & = [\riembase,F^q,{\riem}^q] \in T(\riem'); \\ 
  \mathfrak{B}^{-1}(q)  & = {\riem}^q;  \\
  {\Phi} \circ  \mathcal{F}_\tau(q) & = (\mathscr{R}^q,\Omega^q,f^q) \in \widetilde{\mathcal{M}}_P(g,n).  
 \end{align}

  Note that although $F^q$ is not uniquely specified, the following are uniquely determined by $q$ by the Teichm\"uller equivalence relation (Definition \ref{de:Teich_space_and_eq_rel}): 
 \begin{itemize}
     \item the boundary values of $F^q$ on $\partial \riembase$;
     \item the homotopy class rel boundary of $F^q$ as a map from $\riembase$ to ${\riem}^q$;
     \item $\mathscr{R}^q$, $\Omega^q$, and $f^q = F^q \circ \tau$. 
 \end{itemize}
 Thus we may write a specific representative of $[\rho] q$ as 
 $(\riembase,F^{[\rho]q},{\riem}^{[\rho] q})$, where the ambiguity is confined to $F^{[\rho] q}$. This ambiguity can be safely ignored because its boundary values and homotopy class are nevertheless uniquely determined.\\
 
 The following fact, which follows from the Teichm\"uller equivalence relation together with the definition of sewing, will be of use later.  
 \begin{proposition} \label{pr:modular_action_Teich_curve} Fix $k \in \{1,\ldots,n\}$. 
  Let $[\rho] \in \mathrm{PModI}_k(\riembase)$. For fixed $q = [\riembase,F^q,{\riem}^q]$ as above,
  and any representative $\rho$ of $[\rho]$, we have that 
  \begin{enumerate}\normalfont
      \item $\rho$ extends by the identity to the caps $\Omega_1,\ldots,\Omega_k$ of $\mathscr{R}$, and this extension is quasiconformal;
      \item there is a unique biholomorphism $\sigma_{q,[\rho]}:\riem^q  \rightarrow {\riem}^{[\rho] q}$ such that $(F^{[\rho]q})^{-1} \circ \sigma_{q,[\rho]} \circ F^q \circ \rho^{-1}$  is homotopic to the identity rel boundary of $\riembase$; 
      \item $\sigma_{q,[\rho]}$ extends holomorphically to a map
       $\sigma_{q,[\rho]}: {\riem}^q_k \rightarrow {\riem}^{[\rho]q}_k$;
      \item $f_m^{[\rho] q} = \sigma_{q,[\rho]} \circ f_m^q$ for $m=1,\ldots,k$. 
  \end{enumerate}
 \end{proposition}  
 \begin{proof}  
  (1) follows from the fact that $\rho$ is the identity on each boundary curve $\partial_m \riem$ for $m=1,\ldots,k$.  Aside from uniqueness, (2) is a restatement of the Teichm\"uller equivalence relation $(\riem',F^{[\rho]q},\Sigma^{[\rho]q}) \sim (\riem',F^q \circ \rho^{-1}, \Sigma^q)$. Uniqueness follows from the fact that any pair of such biholomorphisms $\sigma$ would have to be homotopic to the identity rel boundary, and in particular would be equal on the boundary. Thus they would be equal. 

  (3) follows from (4). To see that (4) holds, observe that $f_m^q = F^q \circ \tau_m$ and $f_m^{[\rho]q} = F^{[\rho]q} \circ \tau_m$. Thus by (2), $f_m^{[\rho] q}$ and $\sigma_{q,[\rho]} \circ f_m^q$ have the same boundary values for $m=1,\ldots,k$. Since they are both holomorphic, they must be equal. 
 \end{proof}  
 
 \begin{remark}[The reason behind the notation $\sigma_{q,[\rho]}$ as opposed to $\sigma_{q_1,q_2}$ or $\sigma_{q,[\rho]q}$]
   The uniqueness of $\sigma_{q,[\rho]}$ follows from the fact that $F^{[\rho] q} \circ \sigma_{q,[\rho]} \circ F^q \circ \rho^{-1}$ equals the identity on $\partial \riembase$ and hence  
   \[  \left. \sigma_{q,[\rho]} \right|_{\partial{\riem}^q} = \left. F^{[\rho]q}  \circ \rho  \circ (F^q)^{-1}     \right|_{\partial{\riem}^q}.   \]
   Since $\sigma_{q,[\rho]}$ is holomorphic it is uniquely determined by these boundary values. 
   
   On the other hand, $\sigma_{q,[\rho]}$ does depend on $[\rho]$ and not just on $q$ and $[\rho]q$, because it might turn out that there are $[\rho_1],[\rho_2]$ such that $[\rho_1] q = [\rho_2] q$ even though $[\rho_1]\neq [\rho_2]$. In other words, the modular group actions may have fixed points on Teichm\"uller space. This is the reason that we do not---indeed cannot---use the notation $\sigma_{q_1,q_2}$ or $\sigma_{q,[\rho]q}$. 
   
 \end{remark}
\end{subsection}
\end{section}
\begin{section}{Modular invariants and conformal invariants}

\begin{subsection}{Generalized modular invariants on Teichm\"uller space are relative conformal invariants}

 \label{se:modular_invariants_are_conformal_invariants}

 Let ${\riem}_k$ be a genus $g$ bordered Riemann surface with $n-k$ borders, containing $k$ quasidisks $\Omega_1,\ldots,\Omega_k$ whose closures do not intersect. Assume that the quasidisks contain punctures $p_1,\ldots,p_k$. It is assumed as usual that the caps $\Omega_1,\ldots,\Omega_k$ and borders $\partial_{k+1} {\riem}_k,\ldots, \partial_{n} {\riem}_k$ are ordered, and $f_l$, $l=1,\ldots,k$ are riggings mapping $\disk$ into $\Omega_l$ respectively, taking $0$ to $p_l$. For simplicity, we are using the notation of previous sections, even though we have not sewn on caps to the remaining borders. 
 
 We define the configuration space
 \[   \hat{\mathcal{C}}(k,n,g) = \{ ({\riem}_k,\Omega_1,\ldots,\Omega_k,f_1,\ldots,f_k) \}.  \] 

 We can summarize the relation between the many moduli spaces with the following diagram.

\begin{tikzcd}[column sep = large]
 \mathscr{T}(\riem)  \arrow{dd}{/\mathrm{PModI}_n} \arrow[bend left=20]{ddr}{/\mathrm{PModI}_{n-1}} \arrow[bend left=20]{ddrrr}{/\mathrm{PModI}_{0}}   & & &   \\ & & & 
 \\ \begin{array}{l} \mathcal{C}(0,n,g) \\ = \widetilde{\mathcal{M}}_P(g,n) \end{array}  \arrow{r}{/\mathrm{QS}(\partial_n \riem)} &   \mathcal{C}(1,n,g) \arrow{r}{/\mathrm{QS}(\partial_{n-1} \riem)} &  \cdots  \arrow{r}{/\mathrm{QS}(\partial_1\riem)} &   \begin{array}{l} \mathcal{C}(n,n,g)  \\ = \mathcal{M}(g,n) \end{array}
\end{tikzcd}
\medskip

 We say that a quantity $\mathfrak{m}:\hat{\mathcal{C}}(k,n,g) \rightarrow \mathbb{R}$ is conformally invariant if for any biholomorphism $G:{\riem}_k \rightarrow {\riem}_k'$ preserving the ordering of the punctures and borders, 
 \[   \mathfrak{m}(G({\riem}_k),G(\Omega_1),\ldots,G(\Omega_k),G \circ f_1,\ldots, G \circ f_k) =   \mathfrak{m}({\riem}_k,\Omega_1,\ldots,\Omega_k,f_1,\ldots,f_k).     \]
 
 Alternatively, one can think of conformal invariants as functions on the space of equivalence classes 
 \[ \mathcal{C}(k,n,g) = \hat{\mathcal{C}}(k,n,g)/\sim \] 
 where we say 
 \[  ({\riem}_k,\Omega_1,\ldots,\Omega_k,f_1,\ldots,f_k) \sim (G({\riem}_k),G(\Omega_1),\ldots,G(\Omega_k),G \circ f_1,\ldots,G\circ f_k)  \]
 for any biholomorphism $G:{\riem}_k \rightarrow G({\riem}_k)$. Observe that we have $\mathcal{C}(0,n,g)=\widetilde{\mathcal{M}}_P(g,n)$. 

\begin{definition}
    We call the boundaries $\partial_{k+1} \riem,\ldots,\partial_{n} \riem$ {\it free boundaries}.
\end{definition}
The invariants on $\mathcal{C}(k,n,g)$ detect how the boundaries $\partial_1 \riem, \ldots, \partial_k \riem$ are sewn to disks, where as it does not detect the remaining boundaries $\partial_{k+1} \riem, \ldots, \partial_n \riem$.  

 Needless to say variations on this theme are easily obtained. One can for example allow more punctures in the caps, or allow more general surfaces in place of the caps.

 \begin{definition} 
  Let $M:\mathscr{T}(\riem) \rightarrow \mathbb{R}$ be a function on Teichm\"uller space. We say that $M$ is a modular invariant of type $k$ for fixed $k\in \{1,\ldots,n\}$ if  $M([\rho]q) = M(q)$ for all $[\rho] \in \mathrm{PModI}_k(\riem)$ and $q\in \mathscr{T}(\riem).$ 
 \end{definition}

 There is a natural association between functions on ${\mathcal{C}}(k,n,g)$ and $k$-modular invariants obtained by removing the $n-k$ caps. Define the following map: 
 \begin{align*}
   \mathcal{G}_k: \widetilde{\mathcal{M}}_P(g,n) & \rightarrow \mathcal{C}(k,n,g) \\
   (\mathscr{R},\Omega,f) & \mapsto ({\riem}_k,\Omega_1,\ldots,\Omega_k,f_1,\ldots,f_k);
 \end{align*}
 that is, $\mathcal{G}_k$ removes the caps $\Omega_{k+1},\ldots,\Omega_n$ and forgets the mappings $f_{k+1},\ldots,f_n$. The relation between the spaces is given in the following diagram. 
 
\begin{tikzcd}[column sep = large]
 \mathscr{T}(\riem)  \arrow{d}{\Phi \circ \mathcal{F}_\tau} & & &
 \\  \mathcal{C}(0,n,g)  = \widetilde{\mathcal{M}}_P(g,n)   
 \arrow{d}{\mathcal{G}_1} \arrow{dr}{\mathcal{G}_2} \arrow{drrr}{\mathcal{G}_n} && &\\
 \mathcal{C}(1,n,g) \arrow{r}& \mathcal{C}(2,n,g)  \arrow{r}& \cdots  \arrow{r} & \mathcal{C}(n,n,g)  = \mathcal{M}(g,n)  
\end{tikzcd}
\medskip

\begin{theorem}[conformal invariance is modular invariance]  \label{th:conf_inv_and_mod_inv}
 Fix a base surface $\riem'$ of genus $g$ with $n$ ordered boundaries and quasisymmetric parametrizations $\tau_1,\ldots,\tau_n$. 
 
 Given a conformal invariant $\mathfrak{m}$ on $\hat{\mathcal{C}}(k,n,g)$, the function
 \begin{equation} \label{eq:M_from_m}
   M = \mathfrak{m} \circ \mathcal{G}_k \circ \Phi \circ \mathcal{F}_\tau : T(\riem') \rightarrow \mathbb{R}   
 \end{equation}
 is a $k$-modular invariant on $T(\riem')$. Conversely, given a $k$-modular invariant $M$ on $T(\riem')$, the function 
 \begin{equation} \label{eq:m_from_M}
   \mathfrak{m}=  M \circ \left( \mathcal{G}_k \circ \Phi \circ \mathcal{F}_\tau\right)^{-1}: \mathcal{C}(k,n,g) \rightarrow \mathbb{R}   
 \end{equation}
 is a conformally invariant function on $\hat{\mathcal{C}}(k,n,g)$ which is independent of $\tau_{k+1},\ldots,\tau_{n}$. 
\end{theorem}
Note that on the other hand, the correspondence does indeed depend on $\tau_1,\ldots,\tau_k$.

  The fact that $M$ is a $k$-modular invariant follows immediately from the fact that $\mathcal{G}_k$ does not depend on $\Omega_{k+1},\ldots,\Omega_n$ or $f_{k+1},\ldots,f_n$.    
  The converse requires the following lemma. 
   \begin{lemma} \label{le:lift_of_m} Let $\riem'$ be a fixed Riemann surface with $g$ handles and $n$ boundary curves.  If $\mathcal{G}_k \circ \Phi \circ \mathcal{F}_\tau(\riem',F,\riem) = \mathcal{G}_k \circ \Phi \circ \mathcal{F}_\tau(\riem',\tilde{F},\tilde{\riem})$ then there is a $[\rho] \in \mathrm{PModI}_k(\riem')$ such that $[\riem',F,\riem] = [\riem',\tilde{F} \circ \rho^{-1}, \tilde{\riem}]$. 
 \end{lemma}
 \begin{proof}  
   Denote 
   \begin{align*}
       \Phi \circ \mathcal{F}_\tau([\riem',F,\riem]) & = (\mathscr{R},\Omega,f) \\
       \Phi \circ \mathcal{F}_\tau([\riem',\tilde{F},\tilde{\riem}]) & = (\tilde{\mathscr{R}},\tilde{\Omega},\tilde{f}).
   \end{align*}
   Since $\mathcal{G}_k(\mathscr{R},\Omega,f)=\mathcal{G}_k(\tilde{\mathscr{R}},\tilde{\Omega},\tilde{f})$, there is a biholomorphism $\nu:{\riem}_k \rightarrow \tilde{\riem}_k$ such that $\nu(\Omega_l)=\tilde{\Omega}_l$ and
   \begin{equation} \label{eq:lemma_lift_temp}
     \nu \circ f_l =\tilde{f_l}     \ \ \ \text{for}  \ \ \ l=1,\ldots,k.  
   \end{equation}
   Now let $[\hat{\rho}] \in \mathrm{PModI}_k(\riem')$ be such that 
   \[  \left. \nu \circ F \right|_{\partial_l \riem} = \left. \tilde{F}  \right|_{\partial_l \riem}= \left. \tilde{F}  \right|_{\partial_l \riem} \circ {\hat{\rho}}^{-1} \]  for $l=k+1,\ldots,n$;
   we already have this for $l=1,\ldots,k$ by (\ref{eq:lemma_lift_temp}) and the fact that $\hat{\rho}$ is the identity on the boundaries $\partial_1 \riem,\ldots,\partial_k \riem$.  

   Thus $\hat{\rho} \circ \tilde{F}^{-1} \circ \nu \circ F$ is the identity on $\partial \riem'$, so there is a $\bar{\rho} \in \mathrm{PModI}(\riem')$ such that $\bar{\rho} \circ \hat{\rho} \circ \tilde{F}^{-1} \circ \nu \circ F$ is homotopic to the identity rel boundary. Setting $\rho = \bar{\rho} \circ \hat{\rho}$, we get 
   \[ [\riem',F,\riem]  = [\riem',\tilde{F} \circ \rho^{-1}, \tilde{\riem}] \]
   as claimed. 
 \end{proof}

  We return to the proof of Theorem \ref{th:conf_inv_and_mod_inv}. 
\begin{proof}  
  Let $\riem'$ and $\tau$ be fixed as above, and let $M$ be a $k$-modular invariant on $T(\riem')$.
 It follows from Lemma \ref{le:lift_of_m} and $k$-modular invariance of $m$ that (\ref{eq:m_from_M}) is well-defined.  So $\mathfrak{m}$ is defined by 
 \begin{align*}
    \mathfrak{m}: {\mathcal{C}}(k,n,g) & \rightarrow \mathbb{R} \\
    [{\riem}_k,\Omega_1,\ldots,\Omega_k,f_1,\ldots,f_k] & \rightarrow M([\riem',F,\riem])
 \end{align*}
 where $[\riem',F,\riem] \in T(\riem')$
 is any element such that  
 \[  \mathcal{G}_k \circ {\Phi} \circ  \mathcal{F}_\tau([{\riem}',F,\riem]) = [{\riem}_k,\Omega_1,\ldots,\Omega_k,f_1,\ldots,f_k].  \] 
 Observe that that this forces $F \circ \tau_l = f_l$ for $l=1,\ldots,n$. This is crucial for the originally specified functions $f_1,\ldots,f_k$, whereas the arbitrarily obtained $f_{k+1},\ldots,f_n$ do not affect $\mathfrak{m}$.\\ 
Next we show that $\mathfrak{m}$ is independent of $\tau_{k+1},\ldots,\tau_n$. Momentarily we denote it by $\mathfrak{m}_\tau$ to emphasize dependence on $\tau$. Let $\tau_{k+1}',\ldots \tau_n'$ be possibly new parametrizations, and let $\tau'$ denote $\tau_1,\ldots,\tau_k,\tau_{k+1}',\ldots,\tau'_n$. By well-definedness of $\mathfrak{m}_{\tau'}$, we can choose $\phi_{k+1}',\ldots,\phi_n'$ as we like without changing the value of $\mathfrak{m}_{\tau'}$ In particular, we can choose them so that the resulting $f_j'$ satisfy $f_j'=f_j \circ \tau_j \circ (\tau_j')^{-1}$.  In that case we have 
    \[  \mathfrak{m}_\tau(\riem,\Omega_1,\ldots,\Omega_k,f_1,\ldots,f_k)= \mathfrak{m}_{\tau'}(\riem,\Omega_1,\ldots,\Omega_k,f_1,\ldots,f_k).  \]
\end{proof}
\end{subsection}
\begin{subsection}{Modular invariants induced by sections of the vertical cotangent bundle}  \label{se:invariants_from_sections}

 	 Consider the configuration of a Riemann surface $\mathscr{R}$ with punctures $p_1,\ldots,p_n$, and nonoverlapping quasidisks $\Omega_1,\ldots,\Omega_n$ containing $p_1,\ldots,p_n$. Let $\riem$ denote the complement of the closure of $\Omega$ where $\Omega = \Omega_1 \cup \cdots \cup \Omega_n$. Finally let $\alpha_{\riem} \in \mathcal{A}(\riem)$.  We will denote such configurations by $(\mathscr{R},\Omega,\alpha_{\riem})$.

 We will construct modular invariants by the following prescription. First, we define an action on configurations $(\mathscr{R},\Omega,\alpha_{\riem})$. This is based on an idea of Nehari \cite{Nehari_paper}, but requires the notion of overfare to formulate when the boundaries of $\Omega$ are quasicircles. From this action many invariants are derived by the following prescription:
 \begin{enumerate}
  \item Define a section of $T^*_{\mathrm{vert}}\mathcal{T}(\riem)$; 
  \item Determine when the action corresponding to that section is a modular invariant.  
 \end{enumerate}
 The second problem is addressed directly in this section on a formal level, but specific conditions on the section to accomplish this are given in the next section. The problem of consistently defining a section is also addressed there.

 We define the following function on the space of such configurations.  We first need to define the ``overfare'' map 
 \[  \mathbf{O}:\mathcal{A}_{\mathrm{harm}}(\riem) \rightarrow \mathcal{A}_{\mathrm{harm}}(\Omega) \]
 as follows: $\mathbf{O} \alpha$ is the unique one-form on $\Omega$ whose Sobolev--$H^{-1/2}$ boundary values agree with those of $\alpha$ on $\partial \Omega$. The analysis involved in the definition is not self-evident; we summarize the important points here and refer to \cite{Schippers_Staubach_scattering_II} or the survey \cite {Carlosscattering} for proofs. To show that the overfare map is well-defined, one must show that $H^{-1/2}$ boundary values of $\alpha$ exist, and in what sense, and that these boundary values are the $H^{-1/2}$ boundary values of a unique $L^2$ harmonic one-form on $\Omega$. This is the case for quasicircles \cite{Schippers_Staubach_scattering_II}. Furthermore the overfare is a bounded operator for quasicircles when restricted to one-forms with zero boundary periods \cite{Schippers_Staubach_scattering_II}. Indeed this is the case for the semi-exact forms (Definition \ref{de:semi_exact}), for which the corresponding overfare will be denoted by $\mathbf{O}^{\mathrm{se}}$.\\

 \begin{definition} \label{de:action} Let $\alpha_{\riem} \in \mathcal{A}(\riem)$. Assume that $\mathrm{Re}(\alpha_{\riem})$ is exact. Define the action  
 	\[ E(\mathscr{R},\Omega,\alpha_{\riem}) = \| \mathbf{O}\,  \mathrm{Re}( \alpha_{\riem}) \|_{\Omega}^2 + \| \mathrm{Re}(\alpha_{\riem}) \|_{\riem}^2.    \]
 \end{definition} 

The function $E$ is conformally invariant in the following sense. If $g:\mathscr{R}' \rightarrow \mathscr{R}$ is a biholomorphism and  $\Omega= g(\Omega')$ and $\riem=g(\riem')$,  then for any $\alpha_{{\riem}'} \in \mathcal{A}({\riem}')$  we have  
\[   E(\mathscr{R},\Omega,\alpha_{{\riem}})= E(\mathscr{R}'
,\Omega',g^*\alpha_{{\riem}}).   \]

 Given a section of 
 $T^{*}_{\textnormal{vert}}\mathcal{T}({\riem})$, we obtain conformal invariants of many different kinds from a modular invariant $M$ of type $k$.  These unify various kinds of invariants in the literature.
 The nature of the invariant depends on the nature of the section. The idea is that a modular invariant of type $k$, ``remembers'' how the disks $\Omega_1,\ldots,\Omega_k$ sit inside the surface ${\riem}_k$, but does not distinguish different embeddings of ${\riem}_k$ into a compact surface. Invariants of type $k=0$ are familiar conformal invariants which only depend on the conformal equivalence class of $\riem$ without ``seeing'' how it sits in another surface.\\

 \begin{theorem} \label{th:modular_invariant_field}  Fix a base surface $\riem'$ of type $(g,n)$ with boundary parametrizations $\tau = (\tau_1,\ldots,\tau_n)$. Fix $k \in \{1,\ldots,n\}$. Let $s$
 be a field of one-forms over $T(\riem')$ with exact real parts, satisfying  
 \begin{equation} \label{eq:field_oneforms_invariance_condition}
   \sigma_{q,[\rho]}^* s([\rho] q) = s(q), \ \ \forall q \in \mathscr{T}(\riembase), \ \ \forall [\rho] \in \mathrm{PModI}_k(\riembase).   
 \end{equation}
Then the function
  \begin{align*}
    M(s):\mathscr{T}(\riembase) & \rightarrow \mathbb{R} \\
    [\riembase,F,{\riem}] & \mapsto M(\mathscr{R},\Omega,s([\riembase,F,{\riem}]))
  \end{align*}
  where 
  \[  [\mathscr{R},\Omega,f] = \mathcal{F}_\tau([{\riembase},F,\riem]) \]
  is a modular invariant of type $k$. Furthermore, $M(s)$ is independent of $\tau_{k+1},\ldots,\tau_{n}$. 
 \end{theorem} 
 \begin{proof} Given $[\rho] \in \mathrm{PModI}_k(\riem)$ 
 we have, denoting $q = [\riembase,F,\riem]$, 
 \begin{align*}
  M(s)([\rho] [\riembase,F,\riem]) & = M(s)( [\riembase, F \circ \rho^{-1}, \riem] ) = E(\mathscr{R}^{[\rho]q},\Omega^{[\rho]q},s([\rho]q)) \\
  & = E(\mathscr{R}^q,\Omega^q,\sigma^*_{q,[\rho]} s([\rho]q) = E(\mathscr{R}^q,\Omega^q,s(q) \\
  & = M(s)([\riembase,F,\riem])
 \end{align*}

 To prove the second claim, let $\tau'=(\tau_1,\ldots,\tau_k,\tau_k',\ldots,\tau_n')$. Let $[\rho]$ be an element of $\mathrm{PModI}_k$ whose boundary values are given by 
 \[ \left( \left. \rho\right|_{\partial_1 \riem},\ldots, \left. \rho\right|_{\partial_n \riem} \right) = (\mathrm{Id},\ldots,\mathrm{Id},\tau_{k+1}' \circ \tau_{k+1},\ldots,\tau_n' \circ \tau_n).  \]
 Then $\mathcal{F}_{\tau'}([\rho] [\riembase,F,{\riem}]) = \mathcal{F}_{\tau}(\riembase,F,{\riem})$. The claim then follows from the $k$-modular invariance of $s$.  
 \end{proof}

 In the next section, we will see how condition (\ref{eq:field_oneforms_invariance_condition}) is satisfied by choosing fields of one-forms whose real parts vanish on the boundaries $\partial_{k+1} \riem,\ldots, \partial_n \riem$. Using this we will give examples of fields of one-forms satisfying (\ref{eq:field_oneforms_invariance_condition}). The resulting invariants are also monotonic.
\end{subsection}

\begin{subsection}{Constructing sections and conditions for modular invariance} 
 
We now describe how to consistently specify sections of the vertical cotangent bundle, in such a way that the conditions of Theorem \ref{th:modular_invariant_field} are satisfied, so that the result is a modular invariant. 
\label{se:sections_from_data} 
 Fix a Riemann surface $\riem$ of genus $g$ with $n$ ordered borders homeomorphic to $\mathbb{S}^1$. All notation is as in Section \ref{se:definitions_notation_etc}.  
 
\begin{definition}
 We say that a curve $\gamma \in \riem$ is a {\it trajectory} of a holomorphic one-form $\alpha \in \mathcal{A}(\riem)$ if for any vector $v_p$ tangent to $\gamma$ at $p \in \riem$ we have $\mathrm{Re}(\alpha(v_p))=0$. We say that a boundary curve $\partial_m \riem$ of $\riem$ is a trajectory if $\alpha$ is pure imaginary on $\partial_m \riem$ in the sense of $H^{-1/2}$ boundary values.   
\end{definition}

 The invariants are specified by the principal parts at the punctures in $\Omega_1,\ldots,\Omega_k$; the periods of the Hodge-conjugates around the free boundary; and the condition that the boundaries are trajectories.  
 
 We state this precisely. Let $\mathbb{A}_{r} = \{ z \,: \, r<|z|<1 \}$. 
 \begin{definition}[Section data] \label{de:section_data}
   Fix a whole number $n$ and $k \in \{ 0,\ldots,n \}$.  We say that 
   \[ (\beta,c)=(\beta_1,\ldots,\beta_k,c_{k+1},\ldots,c_n)  \] 
   is { section data} for $n-k$ free boundaries if  
     \begin{enumerate}
    \item $\beta_1,\ldots,\beta_k$ are holomorphic one-forms on $\disk \backslash \{0 \}$, such that 
    \begin{enumerate}
     \item $\beta_l \in \mathcal{A}(\mathbb{A}_{r})$ for some $0<r<1$ for each $l$;
     \item $\beta_l$ has at most a pole at $0$ for $l=1,\ldots,k$; 
     \item the residue of $\beta_l$ is {real} for $l=1,\ldots,k$;
    \end{enumerate}
    \item constants $c_{k+1},\ldots,c_n \in \mathbb{R}$, such that 
    \[ \sum_{l=1}^k 2 \pi i \, \mathrm{Res} (\beta_l;0) + \sum_{l=k+1}^n c_l = 0.  \] 
    In the case that $k=0$, we require that $\sum_{l=1}^n c_l = 0$. 
  \end{enumerate} 
 \end{definition} 
 \begin{theorem} \label{th:section_data}
   Fix a whole number $n$ and $k \in \{ 0,\ldots,n \}$. 
   Let $(\beta,c)$ be section data for $n-k$ free boundaries.
  Given a rigged Riemann surface $(\mathscr{R},\Omega,f)$, there exists a unique one-form $\alpha_{\riem}$ with exact real part such that 
  \begin{enumerate}\normalfont
   \item $\alpha_{\riem}$ has a holomorphic extension to ${\riem}_k$;  
   \item the principal parts of $(f_l)^* \alpha_{\riem}$ and $\beta_l$ agree for $l=1,\ldots,k$;
   \item $\partial_{k+1} \riem,\ldots,\partial_n \riem$ are trajectories of $\alpha_{{\riem}}$;
   \item $\int_{\partial_l \riem} \ast \alpha_{{\riem}} = c_l$ for $l=k+1,\ldots,n$. 
  \end{enumerate}
 \end{theorem}
 \begin{proof} 
  We prove existence.  Let $\beta$ be a holomorphic one-form on $\mathscr{R}$, with possible poles at $p_1,\ldots,p_n$, such that $f_l^*\beta$ and $\beta_l$ have the same principal parts at $0$ for $l=1,\ldots,k$. Assume further that $\beta$ has only simple poles at $p_l$ with residues $\mathrm{Res}(\beta;0)=c_l$ for $l=k+1,\ldots,n$. 

  By \cite[Theorem 3.18]{Schippers_Staubach_scattering_II}, there is a unique solution to the following boundary value problem for harmonic one-forms:
  \begin{itemize}
      \item $\gamma= - \mathrm{Re}(\beta)$ on $\partial_l \riem$ for $l=k+1,\ldots,n$;
      \item $\int_{\partial_l \riem} \ast \gamma =0$ for $l=k+1,\ldots,n$;
      \item $\gamma + \mathrm{Re}(\beta)$ is exact. 
  \end{itemize}
  Since the data is real, the solution is also real. Let $\hat{\alpha} = \gamma + i \ast \gamma$ be the unique analytic completion of $\gamma$. One can verify that $\hat{\alpha}+\beta$ {has exact real part and} satisfies conditions (1)-(3). To verify (4), observe that since $\partial_l \riem$ is a trajectory of $\hat{\alpha}$ for $l=k+1,\ldots,n$, 
  \[  \ast \hat{\alpha} = \mathrm{Re} (\ast \hat{\alpha}) -i \mathrm{Re} (\hat{\alpha}) = \mathrm{Re} (\ast\hat{\alpha}), \]
  so the boundary integrals of $\ast \hat{\alpha}$ satisfy
  \[  \int_{\partial_l \riem} \ast \hat{\alpha} = \int_{\partial_l \riem} \mathrm{Re} (\ast \hat{\alpha}) = \int_{\partial_l \riem} \gamma ={ - \int_{\partial_l\riem} \mathrm{Re} (\beta )} \]
  for $l=k+1,\ldots,n$, which establishes (4).
 \end{proof}
 From this we obtain 
 \begin{theorem} \label{th:conformal_invariant_from_field}
 Fix a whole number $n$ and $k \in \{ 0,\ldots,n \}$. 
   Let $(\beta,c)$ be section data for $n-k$ free boundaries. 
 
  There exists a unique field of one-forms $s_{\beta,c}$ such that $\alpha_{{\riem}^q} = s_{\beta,c}(q)$ has exact real part and has the properties of {\textnormal {Theorem \ref{th:section_data}}}; that is
  \begin{enumerate}\normalfont
   \item $\alpha_{{\riem}^q}$ has a holomorphic extension to ${\riem}^q_k$;  
   \item the principal parts of $(f^q_l)^* \alpha_{\riem^q}$ and $\beta_l$ agree for $l=1,\ldots,k$;
   \item $\partial_{k+1} \riem^q,\ldots,\partial_n \riem^q$ are trajectories of $\alpha_{{\riem}^q}$;
   \item $\int_{\partial_l \riem} \ast \alpha_{{\riem}^q} = c_l$ for $l=k+1,\ldots,n$. 
  \end{enumerate}
  Furthermore, $M(s_{\beta,c})$ is a $k$-modular invariant.  
 \end{theorem}
 \begin{proof}
  The fact that this is a $k$-modular invariant follows from Theorem \ref{th:modular_invariant_field} so long as we can show that $s_{\beta,c}$ is invariant. To see this, let $[\rho] \in \mathrm{PModI}_k(\riem)$ and fix $q =[\riem,F^q,{\riem}^q] \in \mathscr{T}(\riem)$. We need only show that $\sigma_{q,[\rho]}^* \alpha_{{\riem}^{[\rho]q}} = \alpha_{{\riem}^q}$, which will follow if we can show it has properties (1)-(4). 
  
  By Proposition \ref{pr:modular_action_Teich_curve} part (3), $\sigma_{q,[\rho]}^* \alpha_{[\rho]q}$ has a holomorphic extension to ${\riem}^q_k$, so (1) holds. (2) follows directly from Proposition \ref{pr:modular_action_Teich_curve} part (4). Finally conditions (3) and (4) are invariant under pull-back by a conformal map. This proves the claim.
 \end{proof} 
 \begin{remark}
  The condition that the residues of $\beta_l$ are pure imaginary guarantees that the boundary periods of $\mathrm{Re}( \alpha_{\riem})$ are zero, since 
  \[  \mathrm{Re} \left( \int_{\partial \Omega}  \alpha_{\riem} \right) = \mathrm{Re} \left( \int_{\partial \disk}  f^* \alpha_{\riem} \right) =   \mathrm{Re} \left(\int_{\partial \disk} \beta_l \right) =0.  \]
  This is necessary for $\mathrm{Re} (\alpha_{\riem})$ to be exact.
 \end{remark}

 \begin{remark}
     Although the proofs of Theorems \ref{th:section_data} and \ref{th:conformal_invariant_from_field} given here are brief, they rely on extensive preparatory analytic results proven elsewhere; namely, the bounded overfare theorem which shows that $\mathbf{O}$ is well-defined for quasicircles. Without resolution of this problem, one cannot define the invariants on Teichm\"uller space.
   Geometrically, they also rely on the relation between the rigged moduli space with the Teichm\"uller space.
 \end{remark}

 By Theorem {\ref{th:modular_invariant_field}}, we can view these invariants $M(s_{\beta,c})$ as functions on $\hat{\mathcal{C}}(k,n,g)$.  We will henceforth write the conformal invariants as 
 \[   \mathfrak{m}_{\beta,c}({\riem}_k,\Omega_1,\ldots,\Omega_k,f_1,\ldots,f_k). \]  

The definition of $E$ immediately gives the following result.
\begin{theorem} \label{th:postivity}
 Fix a whole number $n$ and $k \in \{ 0,\ldots,n \}$. 
   Let $(\beta,c)$ be section data for $n-k$ free boundaries. Then 
   $\mathfrak{m}_{\beta,c} \geq 0$ for all $({\riem}_k,\Omega_1,\ldots,\Omega_k,f_1,\ldots,f_k)$.
\end{theorem}
The directness with which the positivity can be seen belies its importance. Indeed, through this construction and variations, the positivity of the energy $E$ can be used to derive many non-trivial inequalities in complex analysis. Examples will be seen in Section \ref{se:examples} ahead; more examples using this method can be found in the originating paper of Nehari \cite{Nehari_paper}, as well as \cite{Sch_confinv_Janal,Schippersinvariants,Minda_confinv}.

We now derive a contour integral expression for $E$, under the additional assumption that $\alpha_{\riem}$ is the restriction of a holomorphic one-form $\alpha$ on an open neighbourhood of the closure of $\riem$, which has zero periods around each boundary curve. Let $h$ be the primitive of $\mathrm{Re} (\alpha)$. We associate a singular harmonic function $h^\dagger$ to $h$ by  
\[ h^\dagger = h - \mathbf{O} \left( \left. h \right|_{\riem} \right).  \]
\begin{remark}
The geometric interpretation of $h^\dagger$ is that it is the unique harmonic function on $\Omega$ with the same singularity as $h$ but with zero boundary values on $\partial \Omega$.
\end{remark}

\begin{theorem}  \label{th:its_Nehari} 
 Assume that $\alpha_{\riem} \in \mathcal{A}(\mathscr{R})$ has exact real part, with primitive $h$.  Then 
 \[    E(\mathscr{R},\Omega,\alpha_{\riem})= - \int_{\partial \Omega} h \ast dh^\dagger.     \]
\end{theorem}

\begin{remark} \label{re:limiting integral} If any of the components of $\partial \Omega$ are not rectifiable, then we need a definition of the contour integral. In the case that $\partial \Omega$ consists of quasicircles, we can define 
\[  \int_{\partial \Omega} h \ast dh^\dagger = \lim_{r \rightarrow 1} \int_{\Gamma_r} h \ast dh^\dagger \]
where $\Gamma_r = \Gamma_{1,r} \cup \cdots \Gamma_{n,r}$ is the collection of analytic curves $\Gamma_{k,r} = f_k(|z|=r)$ where $f_k$ is a conformal map from $\disk$ onto $\Omega_k$. 
This limit exists for $h$ of bounded Dirichlet energy and $h^\dagger$ $L^2$ on a collared neighbourhood of $\partial \Omega$, and is independent of the choice of conformal maps $f_k$. 
For details see \cite{Schippers_Staubach_scattering_I}.

\end{remark}
\begin{proof}
 First observe that since overfare is real linear, we have 
 \[  \mathbf{O} \overline{h} = \overline{\mathbf{O}h} \]
 for any harmonic $h$, and 
 \[  \mathbf{O}  (\overline{\alpha_{\riem}}) = \overline{\mathbf{O}(\alpha_{\riem})}.     \]

 Set $u = \mathbf{O}(\left. h \right|_{\riem}) = h - h^\dagger$ below, and observe that $du \in \mathcal{A}(\Omega)$.

 To illustrate the derivation more clearly, we first assume that all the curves are smooth.
  In this case we have
 \begin{align*}
  E(\mathscr{R},\Omega,\alpha_{\riem}) & = \iint_{\riem} dh \wedge \ast dh + \iint_{\Omega} du \wedge \ast du \\ 
  & = \int_{\partial \riem} h \ast dh + \int_{\partial \Omega} (h-h^\dagger) \ast d(h-h^\dagger) \\
  & = - \int_{\partial \Omega} h \ast dh^\dagger,
 \end{align*}
 where we have used the fact that $h^\dagger = 0$ on $\partial \Omega$ and taken into account the change of orientation between $\partial \riem$ and $\partial \Omega$. 

  Now assume that the curves are arbitrary quasicircles, so that the integral over $\partial \Omega$ is replaced with a limiting integral over curves $\Gamma_{k,r}$ in the final integral as described in Remark \ref{re:limiting integral}.  We also approximate the curves from inside $\riem$; let $\gamma_{r}$ be the level set of Green's function $G_{\riem}(p,z) = -\log{r}$ for $p$ fixed in $\riem$. For $r$ sufficiently close to $1$ this consists of  $n$ analytic curves $\gamma_{k,r}$ each homotopic to $\partial_k \Omega$. Endow these with the same orientation as $\partial \riem$.  The computation above using Stokes' theorem now produces 
 \begin{align*}
  E(\mathscr{R},\Omega,\alpha_{\riem})  
  & = - \int_{\partial \Omega} h \ast dh^\dagger + \lim_{r \nearrow 1} \sum_{k=1}^n \left( \int_{\gamma_{r,k} }  h \ast dh + \int_{\Gamma_{r,k} }  h \ast dh \right). 
 \end{align*}

 If $k$ corresponds to a non-free boundary curve, then by part (1) of Theorem \ref{th:conformal_invariant_from_field}, we get 
 \begin{equation*}
   \lim_{r \nearrow 1} \left( \int_{\gamma_{r,k} }  h \ast dh + \int_{\Gamma_{r,k} }  h \ast dh \right) = \lim_{r \nearrow 1} \iint_{B_r} dh \wedge \ast dh   
 \end{equation*}
 where $B_r$ is the region bounded by $\Gamma_{r,k}$ and $\gamma_{r,k}$. Since quasicircles have measure zero, and $h$ is harmonic on an open neighbourhood of $\partial_k \riem$,  the limit of the right hand side is zero. 

 If $k$ corresponds to a free boundary curve, then $h$ and $h - {h}^\dagger$  vanish on $\Gamma_k$; using the fact that $\ast dh$ and $h$ have finite Dirichlet energy and the Anchor Lemma \cite[Lemma 3.15]{Schippers_Staubach_scattering_I} we get that both integrals over $\gamma_{r,k}$ and $\Gamma_{r,k}$ vanish in the limit.  
\end{proof}
\begin{remark}
 The right hand side of \ref{th:its_Nehari} is the monotonic quantity of Nehari \cite{Nehari_paper,Schippersinvariants}.  
\end{remark}
Let $H$ and $H^\dagger$ be the (possibly-multivalued) holomorphic completions of $h$ and $h^\dagger$ respectively.  Observing that $H-H^\dagger$ is single-valued, and applying the Cauchy-Riemann equations to the local expressions for $H$ and $H^\dagger$, we can then rewrite this formula as follows:
\begin{equation} \label{eq:contour_integral_analytic}
 E(\mathscr{R},\Omega,\alpha_{\riem}) =  \mathrm{Re}\left( \frac{1}{i} \int_{\partial \Omega} (H^\dagger-H) \frac{\partial H}{\partial z} \,dz \right),   
\end{equation}
which will be convenient for derivation of explicit formulas. \\

We require a further identity comparing the values of $M$ for fixed $\mathscr{R}$ and varying $\Omega$, in terms of contour integrals. Let $\Gamma = \Gamma_1 \cup \cdots \cup \Gamma_n$ be a collection of analytic simple closed curves in $\riem$, which do not intersect, and are such that each curve $\Gamma_k$ is homotopic to $\partial_k \riem$. Let ${\riem}_\Gamma$ denote the subset of $\riem$ bounded by $\Gamma$ and $\partial \riem$, and $\Omega_{\Gamma}$ denote the region bounded by $\Gamma$ and $\partial \Omega$. We then obtain 
\begin{align*}
 E(\mathscr{R},\Omega,\alpha_{\riem}) & = \|  \mathrm{Re} (\alpha_{\riem} )\|^2_{{\riem}_{\Gamma}} + \|\mathrm{Re} (\alpha_{\riem}) \|^2_{\Omega_{\Gamma} \backslash \Omega} + \| \mathbf{O}^{\mathrm{se}} \, \mathrm{Re} ( \alpha_{\riem}) \|^2_{\Omega} \\
 & = \|  \mathrm{Re} (\alpha_{\riem} )\|^2_{{\riem}_{\Gamma}} + \int_{\Gamma} h \ast dh - \int_{\partial \Omega} h \ast dh + \int_{\partial \Omega} u \ast du \\ 
 & =  \|  \mathrm{Re} (\alpha_{\riem}) \|^2_{{\riem}_{\Gamma}} + \int_{\Gamma} h \ast dh - \int_{\partial \Omega} h \ast dh^\dagger.
\end{align*}
Setting 
\[   C_\Gamma = \|  \mathrm{Re} ( \alpha_{\riem} )\|^2_{{\riem}_{\Gamma}} + \int_{\Gamma} h \ast dh \]
we get 
\begin{equation} \label{eq:renormalized_Nehari} 
 E(\mathscr{R},\Omega,\alpha_{\riem}) = C_\Gamma - \int_{\partial \Omega} h \ast dh^\dagger.
\end{equation}

  In the case that the curves are not smooth, the integral in (\ref{eq:renormalized_Nehari}) is replaced by a limiting integral, and  the computation is justified as in the proof of Theorem \ref{th:its_Nehari}.  Similarly for the proof of Theorem \ref{th:monotonicity} below.

Using this, we can prove the following monotonicity theorem, which could be viewed as a law of entropy.   
\begin{theorem}[Monotonicity of $E$]  \label{th:monotonicity}
  Given configurations $(\mathscr{R},\Omega_k)$ for $k=1,2$ such that $\Omega_2 \subset \Omega_1$ and the boundaries do not intersect.   
  Let $\alpha$ be holomorphic on $\mathscr{R}$ and let $\alpha$ denote $\left. \alpha \right|_{{\riem}_k}$ for both $k=1$ and $k=2$.  Let $h_1$, $h_2$ denote the function $h^\dagger$ induced on $\Omega_1$ and $\Omega_2$ respectively   Then 
  \[  E(\mathscr{R},\Omega_2,\alpha) - E(\mathscr{R},\Omega_1,\alpha) = - \int_{\partial \Omega_2} h_1 \ast d h_2.   \]
  In particular, 
  \[   E(\mathscr{R},\Omega_2,\alpha) - E(\mathscr{R},\Omega_1,\alpha) \geq 0.  \]
\end{theorem}
\begin{proof}
 Set $I = M(\mathscr{R},\Omega_1,\alpha) - \int_{\partial \Omega_2} h_1 \ast d h_2 - M(\mathscr{R},\Omega_2,\alpha)$ and compute using (\ref{eq:renormalized_Nehari}) and the fact that $q_1 = 0$ on $\partial \Omega_1$ that 
 \begin{align} \label{eq:monotonicity_temp1} 
  I & = C_\Gamma - \int_{\partial \Omega_1} h \ast d h_1 - \int_{\partial \Omega_2}  h_1 \ast d h_2 - C_\Gamma +\int_{\partial \Omega_2} h \ast d h_2 \nonumber \\
  & =  \int_{\partial \Omega_1} (h_1-h) \ast dh_1 - \int_{\partial \Omega_2} (h_1 - h) \ast d h_2.
 \end{align}
 Analyzing the first term again using $h_1=0$ on $\partial \Omega_1$, we get
 \begin{align*}
  \int_{\partial \Omega_1} (h_1 - h) \ast dh_1 & = \int_{\partial \Omega_1} (h_1 - h) \ast dh_1 -
  \int_{\partial \Omega_1} h_1 \ast d(h_1 -h) \\
  & = \int_{\partial \Omega_2} (h_1 - h) \ast dh_1 -
  \int_{\partial \Omega_2} h_1 \ast d(h_1 -h)
 \end{align*}
 where in the second inequality we have applied Green's identity using the fact that $h_1-h$ and $h_1$ are harmonic on $\Omega_1 \backslash \Omega_2$. Inserting this back into (\ref{eq:monotonicity_temp1}) we obtain that 
 \begin{align*}
  I = \int_{\partial \Omega_2} (h_1 - h) \ast d(h_1-h_2) - \int_{\partial \Omega_2} (h_1-h_2) \ast d(h_1-h) 
 \end{align*}
 which is zero by Green's identity using the fact that $h_1-h$ and $h_1-h_2$ are harmonic on $\Omega_2$.  This proves the first claim. 
 
 The second claim follows by observing that by an identical computation to that in Theorem \ref{th:its_Nehari}, if we set $v$ to be the unique harmonic function on $\Omega_2$ whose boundary values equal $h_1$ on $\partial \Omega_2$ (namely $h_1-h_2$), then   
 \[  - \iint_{\partial \Omega_2} h_1 \ast dh_2 = \iint_{\Omega_2} dv \wedge \ast dv + \iint_{\Omega_1 \backslash \Omega_2} dh_1 \wedge \ast dh_1 \geq 0.  \]
\end{proof}

 We then ask, when will the modular invariants be monotonic? This requires an extra condition, because it is not immediately clear that 
\begin{theorem}  \label{th:monotonic_if_conditions_hold}
  Fix a whole number $n$ and $k \in \{ 0,\ldots,n \}$. 
   Let $(\beta,c)$ be section data for $n-k$ free boundaries. 
 Assume that for some $\riem$ the Teichm\"uller space of $\mathscr{R}$ is trivial, and choose a fixed surface $\mathscr{R}$ to represent.  Assume that for any $\riem$ and $f_1,\ldots,f_n$, there is an extension of $\alpha_{\riem}$ to $\mathscr{R}$, and that this extension is independent of $\riem$ and $f_1,\ldots,f_n$.  Then 
   $\mathfrak{m}_{\beta,c} \geq 0$ for all $({\riem}_k,\Omega_1,\ldots,\Omega_k,f_1,\ldots,f_k)$ is monotonic, in the sense that if $\Omega_m' \subset \Omega_m$ for $m=1,\ldots,k$, then 
   \[  \mathfrak{m}_{\beta,c}({\riem}_k,\Omega_1,\ldots,\Omega_k,f_1,\ldots,f_k) \geq \mathfrak{m}_{\beta,c}({\riem}_k,\Omega_1',\ldots,\Omega_k',f_1',\ldots,f_k'). \]
\end{theorem}
\begin{proof}  
This follows from Theorem \ref{th:monotonicity}. 
\end{proof}

This can occur for example when $\mathscr{R}$ is the once, twice, or thrice punctured sphere.

\end{subsection}
\begin{subsection}{The case of equality} \label{se:equality}
 
 In this section we discuss in what cases the lower bound of zero is attained for the functionals $\mathfrak{m}_s$ described in the previous section. For a given invariant, the extremal configuration lies outside of the moduli space. The extremal case of the inequalities provides important context and motivation.  To treat this situation properly would lengthen the paper considerably, so we will confine ourselves to a sketch; nothing in the rest of the paper depends logically on this section.

 The lower bound of $E$, and hence $\mathfrak{m}_s$ given by section data $(\beta,c)$, is zero.  
 This lower bound is not attained on the configuration space $\hat{C}(k,n,g)$, but if the functional is extended to a larger space, equality can be attained. In general, the lower bound is attained for mappings which fill out $\mathscr{R}$ in the sense that their complement is a measure zero set. Furthermore, the set consists of a network of analytic arcs on which $\mathrm{Re} (\alpha) =0$.  
 
 This is an instance of a general principle in 
 complex analysis, due to M. Schiffer, which says the following. Given a continuous functional on a class of conformal mappings into a target domain maps, the extremal is a map onto the target minus arcs of trajectories of a quadratic differential \cite{Pommerenkebook}.  A quadratic differential is a meromorphic $2$-differential, that is, a symmetric tensor locally given in coordinates by $Q(z)dz^2$ where $Q$ is a meromorphic function of the coordinate $z$ \cite{Lehto_book,Pommerenkebook}.  An example of such a quadratic differential is $\alpha(z)^2$ where $\alpha$ is a meromorphic one-form. 
 
 Trajectories of a quadratic differential are images of curves $t \mapsto \gamma(t)$ for $t$ in an interval $I$ which satisfy 
 \[ Q(\gamma(t))\dot{\gamma}(t) \leq 0.  \] 
 Strictly speaking the sets themselves are the trajectories, that is, they are given neither a parametrization nor a direction. The set consists of a network of piecewise analytic arcs. 

 Domains of this type are not in the configuration space, so we are immediately confronted by the problem of extending the functionals to a larger space. We adopt a minimal extension which suffices to include the case of equality.

 For ${\riem}_k$ as above we say that a configuration $({\riem}_k,f_1,\ldots,f_k,\Omega_1,\ldots,f_k)$ is {\bf simple} if $f$ and $\Omega_1,\ldots,\Omega_k$ satisfy the following weaker conditions:
 \begin{enumerate}
  \item $f_m:\disk \rightarrow {\riem}_k$ are conformal maps onto $\Omega_m \subset {\riem}_m$ for $m=1,\ldots,k$; 
  \item the images $\Omega_m$ do not intersect;
  \item there is a quadratic differential $Q$ on ${\riem}_k$ such that
  \begin{enumerate}
   \item the boundary $\partial {\riem}_k$ is a trajectory of $Q$;
   \item the poles of $Q$ are in $\Omega_1,\ldots,\Omega_k$;
   \item for $m=1,\ldots,k$, $\partial \Omega_m$ is a network of trajectories of $Q$. 
  \end{enumerate} 
 \end{enumerate} 
 Let $\tilde{C}(k,g,n)$ be the union of $\hat{\mathfrak{C}}(k,g,n)$ with the simple domains.
 
 The contour integrals now extend to simple domains. That is, the function $q^\dagger$ is obtained from $q$ by restricting it to the piecewise analytic curves $\partial \Omega$ and solving the Dirichlet problem on each $\Omega_m$. Note that there is no need to invoke the overfare theorem. At some points on the boundary it might be possible to approach from two sides of the analytic arc, and the function $q^\dagger$ has possibly different limiting values. Using Carath\'eodory's theory one has well-defined boundary values of $q^\dagger$ on the ideal boundary, and can define the contour integral 
 \[  \int_{\partial \Omega} q d q^\dagger  \]
 (see \cite[Appendix]{Schippersinvariants}). 
 The functional $\mathfrak{m}_{\beta,c}$ thus extends to $\tilde{C}(k,g,n)$. Equivalently one may directly use Definition \ref{de:action}. 
 
 \begin{remark} \label{re:continuous_ext_of_functionals}
  Given local coordinates in a neighbourhood of the punctures, it follows from the contour expression of Theorem \ref{th:its_Nehari} that the functional $\mathfrak{m}_{\beta,c}$ and its extension are a continuous function of finitely many Taylor coefficients of the mapping functions $f_1,\ldots,f_n$.  For a fixed surface ${\riem}_k$, these coefficients are continuous in the topology of uniform convergence of compact sets by the Cauchy integral formula. Thus $\mathrm{m}_{\beta,c}$ is continuous in this topology for the fixed surface. Note that this is not the topology of Teichm\"uller space (see e.g. \cite[Section 4.4]{Lehto_book}). 
 \end{remark}

\begin{theorem} \label{th:equality} Let $s$ be the field of one-forms on  
 $T(\riem')$ specified by the section data $(\beta,c)$.   Let $M(s)$ be the $k$-modular invariant and 
 let $\mathfrak{m}_s$ be the associated conformal invariant on ${\tilde{\mathcal{C}}}(k,n,g)$. Then $\mathfrak{m}_s({\riem}_k,\Omega_1,\ldots,\Omega_k,f_1,\ldots,f_k)=0$
 if and only if the complement of $\Omega_1 \cup \cdots \cup \Omega_k$ in ${\riem}_k$ is a measure zero set, and the boundaries of $\Omega_1,\ldots,\Omega_k$ are 
 trajectories of $\alpha_{{\riem}}$.  
\end{theorem}
\begin{proof}
    Since $\alpha_{{\riem}'_k}$ is non-zero, if equality holds, we must have that the complement ${\riem}'$ is a set of measure zero.  We also need that the overfares of $\mathrm{Re} \alpha_{{\riem}'}$ are zero. This the case if $\mathrm{Re} \alpha_{{\riem}'}$ is zero on $\partial \Omega_l'$ for $l=1,\ldots,k$ as claimed. 
\end{proof}

One could ask whether there are more extremal maps, which would be obtained by including less regular mappings. There are none, by an argument using well-established techniques which would run like this. The functionals could be extended continuously to include mappings satisfying (1) and (2) but not necessarily (3), e.g. by expressing the functional in terms of limiting contour integrals or directly in terms of Taylor coefficients in coordinates (see Remark \ref{re:continuous_ext_of_functionals}). 
Applying the variational techniques of Schiffer, one could then show that the extremals were obtained by mappings in the class $\tilde{C}(k,n,g)$.  
  
\end{subsection}
\begin{subsection}{The Schiffer operators} \label{se:Schiffer_operators}
   In anticipation of examples constructed in Section \ref{se:examples}, we define the Schiffer operators. These are integral operators whose integral kernels are certain singular bi-differentials.  We consider the compact surface $\hat{\mathscr{R}}$ obtained by filling in the punctures. 
 \begin{definition}
  For a compact Riemann surface $\hat{\mathscr{R}}$ we let 
  {\it Green's function} $\mathscr{G}(w,w_0;z,z_0)$ be the unique harmonic function in $w$, for $w\neq z$, $w \neq z_0$, such that for a local coordinate $\phi$ near $z$
  \[  \mathscr{G}(w,w_0;z,z_0) + \log|\phi(w)-\phi(z)|   \]
  is harmonic in a neighbourhood of $z$, and for a local coordinate $\psi$
  \[  \mathscr{G}(w,w_0;z,z_0) - \log|\phi(w)-\phi(z_0)|    \]
  is harmonic in a neighbourhood of $z_0$; and furthermore $\mathscr{G}(w,w_0;z,z_0)$ vanishes at $w_0$. This is also harmonic in $z$ away from $w$ and $z_0$, and enjoys a number of symmetry properties.
  The  {\it Schiffer kernel} is defined by 
 \[  L_{\hat{\mathscr{R}}}(z,w) =  \frac{1}{\pi i} \partial_z \partial_w \mathscr{G}(w,w_0;z,z_0),    \]
 and the {\it Bergman kernel} is given by
 \[  K_{\hat{\mathscr{R}}} (z,w) = - \frac{1}{\pi i} \partial_z \overline{\partial}_{{w}} \mathscr{G}(w,w_0;z,z_0).    \]
 
 \end{definition}

  The kernel functions satisfy the following:
 \begin{enumerate} 
  \item[$(1)$] $L_{\hat{\mathscr{R}}}$ and $K_{\hat{\mathscr{R}}}$ are independent of $z_0$ and $w_0$.  
  \item[$(2)$] $K_{\hat{\mathscr{R}}}$ is holomorphic in $z$ for fixed $w$, and anti-holomorphic in $w$ for fixed $z$.
  \item[$(3)$] $L_{\hat{\mathscr{R}}}$ is holomorphic in $w$ and $z$, except for a pole of order two when $w=z$.
\item[$(4)$] $L_{\hat{\mathscr{R}}}(z,w)=L_{\hat{\mathscr{R}}}(w,z)$.  
  \item[$(5)$] $K_{\hat{\mathscr{R}}}(w,z)= - \overline{K_{\hat{\mathscr{R}}}(z,w)}$.  
 \end{enumerate}

From the conformal invariance of Green's function, it follows that the kernels above are conformally invariant. 
 
\begin{definition}\label{def: restriction ops}
Define the restriction operators
 \begin{align*}
  \mathbf{R}_{\Omega}:\mathcal{A}(\hat{\mathscr{R}}) & \rightarrow \mathcal{A}(\Omega) \\
  \alpha & \mapsto \left. \alpha \right|_{\Omega}.
 \end{align*}
 and 
 \begin{align*}
\mathbf{R}_{\riem}:\mathcal{A}(\hat{\mathscr{R}}) & \rightarrow \mathcal{A}(\riem) \\
  \alpha & \mapsto \left. \alpha \right|_{\riem}
 \end{align*}
\end{definition}
     
 It is obvious that these are bounded operators. 

Having the Bergman and Schiffer kernels and the restriction operators at hand, we can now define (special cases of) the Schiffer operators as follows.
\begin{definition} Define 
 \begin{align*}
  {\bf{T}}_{\riem}: \overline{\mathcal{A}(\Omega)} & \rightarrow \mathcal{A}(\riem)  \\
  \overline{\alpha} & \mapsto \iint_{\Omega} L_{\hat{\mathscr{R}}}(\cdot,w) \wedge \overline{\alpha(w)} \ \ \ \ \ \cdot \in \riem
 \end{align*} 
and 
 \begin{align*}
  {\bf{T}}_{\Omega}: \overline{\mathcal{A}(\Omega)} & \rightarrow \mathcal{A}(\Omega)  \\
  \overline{\alpha} & \mapsto \iint_{\Omega} L_{\hat{\mathscr{R}}}(\cdot,w) \wedge \overline{\alpha(w)} \ \ \ \ \ \cdot \in \Omega.
 \end{align*}
 where the distinction between the two operators is in the location of the output variable in the first entry of $L_{\hat{\mathscr{R}}}$. The second integral is interpreted as a principal value integral. Finally define 
\begin{align*}
 {\bf{S}}: \mathcal{A}(\Omega) & \rightarrow \mathcal{A}(\hat{\mathscr{R}}) \\
  \alpha & \mapsto \iint_{\Omega}  K_{\hat{\mathscr{R}}}(\cdot,w) \wedge \alpha(w).
\end{align*}
\end{definition} 
These operators are all bounded. 
\begin{remark} \label{re:S_vanishes_genus_zero} If $\hat{\mathscr{R}}$ has genus zero, then $K_{\hat{\mathscr{R}}}$ is zero. For the same reason $\mathbf{S}= 0$ (indeed, $\mathcal{A}(\hat{\mathscr{R}})$ consists only of the one-form which is identically zero. 
\end{remark}
 
For any operator $\mathbf{M}$, we define the complex conjugate operator by 
 \[  \overline{\mathbf{M}} \overline{\alpha} = \overline{\mathbf{M} \alpha}.  \]
 So for example 
 \[   \overline{\mathbf{T}}_{\riem}:\mathcal{A}(\Omega) \rightarrow \overline{\mathcal{A}(\riem)}. \] 
 
 The restriction operator is conformally invariant by conformal invariance of Bergman space of one-forms. 
 By the conformal invariance of the kernel functions defined above, the operators $\mathbf{T}$ and $\mathbf{S}$ are also conformally invariant. The same holds for the conjugate operators. 

 For some historical context of these operators see \cite{Schippers_Staubach_scattering_III}. 

\end{subsection}
\begin{subsection}{Faber-Tietz forms} \label{se:Faber-Tietz}
 In this section, we show that the field of one-forms obtained from Theorem \ref{th:conformal_invariant_from_field} can be written in terms of certain one-forms which we call Faber-Tietz forms \cite{SchiShi_Faber}, in the case that there are no free boundaries.

 Now assume that none of the boundary curves are free, and we are given data $\beta_l$ as in Theorem \ref{th:conformal_invariant_from_field}. 
 \begin{lemma} \label{le:remove_simple_poles} Fix imaginary numbers $d_1,\ldots,d_n$ such  that $d_1+\ldots+d_n=0$. There is a holomorphic one-form   $\delta$ on $\hat{\mathscr{R}}$ with simple poles at $p_1,\ldots,p_n$ with residues $d_1,\ldots,d_n$ respectively, whose real part has zero periods around the handles.  
 \end{lemma}
 \begin{proof}  A classical existence theorem guarantees 
  the existence of meromorphic one-forms $\delta_l$ with simple poles of residue $-1$ at $p_1$ and residue $1$ at $p_l$, and no other poles \cite[p 51]{Farkas_Kra}. Thus we may find a holomorphic one-form $\delta'$ with the specified poles. There is also a real harmonic one-form $\epsilon$ say on $\hat{\mathscr{R}}$ whose real part has periods are equal to the real parts of the periods of $\delta'$ \cite[p 64]{Farkas_Kra}. Setting $\gamma = \epsilon + i \ast \epsilon$ to be the analytic completion of $\epsilon$, we then have that $\delta=\delta' - \gamma$ has the desired properties. 
 \end{proof}

 Now given data $\beta_1,\ldots,\beta_n$ as in Theorem \ref{th:conformal_invariant_from_field}, with residues $d_1,\ldots,d_n$ respectively, let $\delta$ be the holomorphic one-form in Lemma \ref{le:remove_simple_poles}. Then $\beta_l' = \beta_l -f_l^* \delta$ has zero residue. If we find $\alpha_{\riem}'$ corresponding to the data $\beta_l'$, then $\alpha_{\riem}= \delta + \alpha_{\riem}'$ corresponds to the data $\beta_1,\ldots,\beta_n$.  Thus, we have reduced the problem to finding $\alpha_{\riem}$ when the data $\beta_1,\ldots,\beta_n$ has only poles of zero residue.  

 This problem is solved with the use of the Faber-Tietz forms, which we now define. Let $(\mathscr{R},f,\Omega)$ be a rigged Riemann surface. 
  Let $e^m_l \in \overline{\mathcal{A}(\disk)} \oplus  \cdots \oplus \overline{\mathcal{A}(\disk)}$ (where there are $n$ copies of $\overline{\mathcal{A}(\disk)}$ on the right hand side) be the one-form
 \[  \left( 0,\ldots, 0, m\, \overline{z}^{m-1} d\, \bar{z}, 0, \ldots,0 \right) \]
 where the only non-zero entry is in the $l$-th component of the direct sum. 
 \begin{definition}[Faber-Tietz forms]  Let $m \geq 1$ be an integer and $l \in \{1,\ldots,n \}$. The $m,l$-th Faber-Tietz form with singularity at $p_l$ is 
 \[  \alpha^m_l = \mathbf{T}_{\riem} (f^{-1})^* (e^m_l).  \]
 \end{definition}
 Here, the pull-back $(f^{-1})^* e$ of $e \in \overline{\mathcal{A}(\disk)} \oplus  \cdots \oplus \overline{\mathcal{A}(\disk)}$ refers to the one-form in $\overline{\mathcal{A}(\Omega)}$ whose restriction to $\Omega_k$ is $(f_l^{-1})^* e_l$ where $e_l$ denotes the $l$-th component of $e$.  
 
 By properties of $\mathbf{T}$ (see \cite[Theorem 3.2]{SchiShi_Faber}, $\alpha^m_l$ is in $\mathcal{A}(\riem)$, and  it extends to a meromorphic differential on $\mathscr{R}$ with poles only at the points $p_1,\ldots,p_n$ \cite[Theorem 3.12]{SchiShi_Faber}.   
 
 \begin{theorem} \label{th:Faber-Tietz_have_specified_pole} For any integer $m \geq 1$ and $l \in \{1,\ldots,n \}$, 
  the Faber-Tietz form $\alpha_k^l$ extends holomorphically to $\mathscr{R} \backslash \{p_l\}$.  Furthermore, $\alpha^m_l$ has a pole of order $m+1$ at $p_l$; in coordinates defined by $f_k$, the principal part at $p_l$ is given by 
  \[  f_l^* \alpha^m_l = 
  \left(\frac{m}{\zeta^{m+1}} + a(\zeta )\right) d\zeta \]
  where $a(\zeta)$ is holomorphic at $0$.  
 \end{theorem} 

 This motivates the following transform. Given 
 \begin{equation}\label{ beta form}
  \beta_l = \sum_{m=1}^N \frac{b_m}{\zeta^{m+1}} \, d\zeta,   
 \end{equation}      
 we define 
 \begin{equation}\label{transform of beta form}
  \hat{\beta}_l= \sum_{m=1}^N m {b_m} \overline{\zeta}^{m-1}\, d\overline{\zeta}.
  \end{equation}
 Given data $\beta_1,\ldots,\beta_n$, denoting $\disk^n = \disk \sqcup \cdots \sqcup \disk$ we define 
 \[ \hat{\beta} \in \overline{\mathcal{A}(\disk^n)} \]
 by 
 \begin{equation}\label{defn beta}
   \left. \hat{\beta} \right|_{\Omega_l} = \hat{\beta}_l.  
 \end{equation} 
 We can also think of $\mathcal{A}(\disk^n) = \mathcal{A}(\disk) \oplus \cdots \oplus \mathcal{A}(\disk)$.
 \begin{remark}
     $\hat{\beta}_l$ is the pull-back of $\beta_l$ under $\zeta \mapsto 1/\overline{\zeta}$.  
 \end{remark}

 We have the following result expressing $\alpha_{\riem}$ in (\ref{th:Faber-Tietz_have_specified_pole}) in terms of the Schiffer operator. 
  \begin{theorem} \label{th:field_characterization_Faber}
    Let $\beta_1,\ldots,\beta_n$ satisfy the conditions of {\textnormal{Theorem \ref{th:conformal_invariant_from_field}}} with $k=n$.  Assume that $\beta_l$ have zero residue for $l=1,\ldots,n$. Then the one-form $\alpha_{\riem}$ obtained from Theorem \ref{th:conformal_invariant_from_field} is explicitly given by 
    \[  \alpha_{\riem} = \mathbf{T}_{\riem} (f^{-1})^*\hat{\beta} + \mathbf{S} (f^{-1})^* \overline{\hat{\beta}}.    \]
    If $\mathscr{R}$ is of genus zero, then 
    \[  \alpha_{\riem} = \mathbf{T}_{\riem} (f^{-1})^*\hat{\beta}. \]
  \end{theorem}
  \begin{proof}
   It follows from Theorem \ref{th:Faber-Tietz_have_specified_pole} that $\alpha_{\riem}$ extends holomorphically to $\mathscr{R}$. It was shown in \cite[Theorem 4.7]{Schippers_Staubach_scattering_III} that $\mathbf{T}_{\riem} \hat{\beta} + \overline{\mathbf{S}} \hat{\beta}$ is exact, which proves the claim. It was shown in general that elements of the image of $\mathbf{T}_{\riem}$ have zero boundary periods, which proves the claim in genus zero. 
  \end{proof}
  \begin{remark}
      This theorem also expresses $\alpha_{\riem}$ as a sum of Faber-Tietz forms up to the holomorphic correction term $\mathbf{S}_{\Omega}$. 
  \end{remark}
\end{subsection}
\end{section}
\begin{section}{Examples} \label{se:examples}
\begin{subsection}{Doubly-connected domains} 
 We will consider surfaces $\riem$ with no handles and two boundary curves homeomorphic to $\mathbb{S}^1$.  The first three examples have data with only simple poles, and deal with the cases that there are zero, one, or two free boundaries.  The case of one free boundary relates to the Schwarz lemma, and the case of two free boundaries relates to the module of the annulus. 
 
 In the examples below, we adopt the following numbering convention for the boundaries. After sewing on caps, one obtains a compact Riemann surface of genus zero with two punctures, which thus can be identified with the twice punctured sphere. One of the domains $\Omega_k$ will be unbounded; we choose it to be $\Omega_2$.

 \begin{example}[no free boundaries]  \label{ex_cylindrical_metric_no_free}
  Let $\riem$ be a doubly-connected surface. 
  Let 
  \[  \beta_1 = \frac{dz}{z}  \ \ \ \mathrm{and} \ \ \ 
   \beta_2 = - \frac{dz}{z}.             \]
  One can check that these satisfy the conditions of 
  Theorem \ref{th:conformal_invariant_from_field}. Condition (4) is void, because no constants $c_l$ are required, since no boundary is free.  

  After sewing on caps, we normalize so that $p_1=0$ and $p_2=\infty$, and then $\riem$ is a doubly-connected domain in the sphere, with complements $\Omega_1,\Omega_2$ containing $0$ and $\infty$ respectively. We also have conformal maps $f_l:\disk \rightarrow \Omega_l$ for $l=1,2$.
  Since $\alpha_{\riem}$ is uniquely specified by its pole at $\infty$ and $0$, we obtain
  \[  \alpha_{\riem} = \frac{dz}{z} \]
  since this form has the correct residues.

  Denote $h_l=\left. h^\dagger \right|_{\Omega_l}$ for $l=1,2$.
  We then obtain the following values for the functions $q$, $h_1$, and $h_2$:
  \begin{align*}
   h & = \log{|z|} \\
   h_1 & = \log{|f_1^{-1}(z)|} \\
   h_2 & = - \log{|f_2^{-1}(z)|}.
  \end{align*}  
  In the expression for $h_2$, observe that $f_2^{-1}$ has 
  Laurent expansion
  \begin{equation} \label{eq:f_2_inverse_temp}
   f_2^{-1}(z) = c_1 z^{-1} + c_2 z^{-2} + \cdots
  \end{equation}
  at $\infty$ so that $h_2-h$ is non-singular. 
  
  We then compute using (\ref{eq:contour_integral_analytic})
  \begin{align*}
   \mathfrak{m}_{\beta,c}(\riem,\Omega_1,\Omega_2,f_1,f_2) & = \mathrm{Re} \left( \frac{1}{i} \int_{\partial \Omega_1} (H_1 - H) \frac{\partial H}{\partial z} \,dz \right) + \mathrm{Re} \left( \frac{1}{i} \int_{\partial \Omega_2} (H_2 - H) \frac{\partial H}{\partial z} \,dz \right) \\
   & = \mathrm{Re} \left( \frac{1}{i} \int_{\partial \Omega_1} \log\Big({\frac{f_1^{-1}(z)}{z}}\Big)\, \frac{dz}{z}  \right) - \mathrm{Re} \left( \frac{1}{i} \int_{\partial \Omega_2} \log\Big({z{f_2^{-1}(z)}}\Big)\, \frac{dz}{z}  \right) \\
   & = \mathrm{Re} \left( \frac{1}{i} \int_{\partial \Omega_1} \log\Big({\frac{f_1^{-1}(z)}{z}}\Big)\, \frac{dz}{z}  \right) + \mathrm{Re} \left( \frac{1}{i} \int_{|z|=R}  \log\Big({z{f_2^{-1}(z)}}\Big)\, \frac{dz}{z} \right)
  \end{align*}
  for $R$ sufficiently large, where the last integral is taken counter-clockwise around the origin.  
   Note that $\partial \Omega_2$ is oriented positively with respect to $\infty$, which accounts for the sign change in the final term. Thus we obtain 
  \begin{align*}
   \mathfrak{m}_{\beta,c}(\riem,\Omega_1,\Omega_2,f_1,f_2) & = 2 \pi \left( - \log|f_1'(0)| + \log|f_2'(0)| \right). 
  \end{align*}
  Here by $f_2'(0)$ we mean the coefficient $a_{-1}$ in the Laurent expansion
  \[ f_2(z)= a_{-1}/z + a_0 + a_1 z + \cdots \]
  at $\infty$ (note that $a_{-1}$ equals the coefficient $1/c_1$ in (\ref{eq:f_2_inverse_temp}).  

  The monoticity of this quantity is a manifestation of the Schwarz Lemma. The quantity $\mathfrak{m}_{\beta,c}(\riem,\Omega_1,\Omega_2,f_1,f_2)$ has no upper bound. By Theorem \ref{th:equality} it is zero precisely in the case that the domains map onto disks centred on $0$ and $\infty$ which are bounded by the same circle, since circles centred at $0$ are the curves such that $\mathrm{Re}(dz/z)=0$.  It can also be explicitly checked that $m=0$ in this case.    
 \end{example}
 
 In the next example, we choose one boundary to be free. Thus ${\riem}_1$ is biholomorphic to a punctured disk. We allow the position of the puncture to vary. 
 \begin{example}[one free boundary; hyperbolic reduced module] \label{ex:hyp_reduced_module}
  Let ${\riem}$ be a doubly-connected Riemann surface. We will allow the second boundary to be free. Thus ${\riem}_1$ is conformally equivalent to a once-punctured disk, so we set ${\riem}_1 = \disk \backslash \{a\}$ for some $a \in \disk$, and we have a single conformal map $f=f_1:\disk \rightarrow \Omega_1=\Omega$ such that $f(0)=a$, and we can ignore the conformal map $f_2$. 

  Again choose $\beta_1 = \beta = dz/z$. There is only one constant $c_2$, which must be $-2\pi i$ since $\mathrm{Res}( dz/z;0)= 1$. 
  
  One may check that 
  \[  \alpha_{\riem} = \frac{(1-|a|^2)}{(1-\bar{a}z)(z-a)} \,dz  \]
  satisfies properties (1)-(4) of Theorem \ref{th:conformal_invariant_from_field}.  With these choices we get 
  \begin{align*}
   x(z) & = \log{ \frac{z-a}{1-\bar{a}z}} \\
   x_1(z) & = \log{f^{-1}(z)}.
  \end{align*}
  It is then computed using (\ref{eq:contour_integral_analytic}) that
  \begin{align*}
   \mathfrak{m}_{\beta,s}(\disk,\Omega,f)   
   & = \mathrm{Re} \left( \frac{1}{i} \int_{\partial \Omega} \log{\left(  (1-\bar{a}z) \frac{f^{-1}(z)}{z-a} \right)} \frac{dz}{z} \right) \\
   & = 2 \pi \log{\frac{(1-|a|^2)}{|f'(0)|}  }   
  \end{align*}
  (see \cite[Example 2.1]{Minda_confinv}). 
  The hyperbolic metrics of $\disk$ and $\Omega$ are 
  \[   \lambda_{\disk}(z)^2 |dz|^2 = \frac{|dz|^2}{(1-|z|^2)^2} \ \ \ \ \ \ \  \lambda_{\Omega}(z)^2 |dz|^2 = \frac{|F'(z)|^2}{(1-|F(z)|^2)^2  }  |dz|^2    \]
  where $F=f^{-1}$, then we may write
  \[    \mathfrak{m}_{\beta,c}(\disk,\Omega,f)  = 2 \pi \log{\frac{\lambda_{\Omega}(a)}{\lambda_{\disk}(a)}}.     \]
  Replacing $\disk$ with ${\riem}_1$ we get a manifestly conformally invariant expression. 
  Observe that the ratio is independent of the curvature normalization of the hyperbolic metric.

  Positivity of this quantity is a special case of the Schwarz-Pick lemma, and equality holds if and only if $f$ maps onto $\disk$. 

  We also have a monotonicity property. If we fix a normalization that $\mathscr{R}=\sphere$ with punctures at $1$ and $1/\overline{a}$, we see that $\alpha_{\riem}$ etends to $\sphere \backslash \{a, 1/\overline{a} \}$ independently of the domain $\riem$ and the maps $f_1$ and $f_2$. Thus Theorem \ref{th:monotonic_if_conditions_hold} applies, and we see that 
  \[  \mathfrak{m}_{\beta,c}(\disk,\Omega,f_1) \leq \mathfrak{m}_{\beta,c}(\disk,\Omega',f_1')   \]
  whenever $\Omega' \subset \Omega$. This can also be obtained from the Schwarz lemma. 
 \end{example}
 \begin{remark} In Example \ref{ex_cylindrical_metric_no_free} the punctures need not have been placed at zero; one would get a more complicated expression for the invariant as in Example \ref{ex:hyp_reduced_module}. 
 \end{remark}

 Finally, we consider the case where both boundaries are free. This must result in a conformal invariant in the classical sense: that is, a function which depends only on the domain itself and not how it nests in another domain. Thus it must be a function of the module of a doubly-connected domain, as we will indeed obtain. 
 \begin{example}[both boundaries free, module of annulus]
  Assuming that both boundaries are free, we have $k=0$ and ${\riem}_0=\riem$.
  The section data for the problem consists only of a pair of constants $c_1,c_2$ summing to zero; we choose $c_1=1$, $c_2=-1$. It can be checked that the unique $\alpha_{\riem}$ obtained from Theorem \ref{th:conformal_invariant_from_field} (up to scale) is 
  \[   \alpha_{\riem} = \frac{1}{2\pi} \frac{dz}{z}.  \]
  To see this we need only check that 
  \[  \int_{\partial \Omega_1} \ast \frac{1}{2 \pi} \frac{dz}{z} = \int_{\partial \Omega_1}  - \frac{i}{2 \pi} \frac{dz}{z} =1 = c_1.   \]
  where it is understood that the integral is taken with positive orientation with respect to the domain $\Omega_1$.   
  We then obtain that the function $h$ in Theorem \ref{th:its_Nehari} is
  \[ h(z) = \frac{1}{2\pi} \log{|z|}. \] 

  In order to compute $\mathfrak{m}_{\beta,c}(\riem)$    We may choose $\riem$ to be a standard annulus
  \[  \riem = \{ z \,: \, r < |z| <R \}.  \]
  Computing using either Theorem \ref{th:its_Nehari} or equation (\ref{eq:contour_integral_analytic}) we obtain 
  \begin{align*}
   \mathfrak{m}_{\beta,c}(\riem) & = \int_{\partial \riem} h \frac{\partial h}{\partial n} \,ds \\
   & =  \int_0^{2\pi} \frac{1}{2\pi} \log{R} d\theta - \int_0^{2\pi} \frac{1}{2\pi} \log{r} d\theta = \log{(R/r)}
  \end{align*}
  where $n$ is the unit outward normal and both integrals were taken with respect to the positive orientation on their respective circles.  By conformal invariance, the module of $\riem$ is in general thus just $\log{R/r}$ where $R$ and $r$ are the outer and inner radii of a conformally equivalent standard annulus. 

  Positivity of the module is trivial.  The conditions of Theorem \ref{th:monotonic_if_conditions_hold} are also met, since we see that $\alpha_{\riem}$ is independent of the domain $\riem$ and the mappings $f_1,f_2$. Thus we obtain that if $\riem' \subset \riem$ then
  \[   \mathfrak{m}_{\beta,c}(\riem') \leq \mathfrak{m}_{\beta,c}(\riem). \]
  
  This is a special case of Gr\"otzsch's lemma; see e.g. \cite[pp 22--23]{Jenkinsbook}, where Gr\"otzsch's lemma is proven using a variant of the length/area method.  Beurling's theorem  \cite[Section 4-9]{Ahlfors_conf_inv} connects the definition of the conformal invariant using extremal length with the definition of the conformal invariant given here in terms of Dirichlet energy. 
 \end{example}

\end{subsection}
\begin{subsection}{Multiply-connected domains in the sphere} 
 We consider multiply-connected planar domains, in particular the case that all boundaries are free and none are free.
 The examples in this section also can be specialized to doubly-connected domains. The first example ahead would reproduce the module of the annulus, while the second is not contained in the list in the previous section. 

 In this section $\riem$ is a Riemann surface with no handles and bordered by $n$ curves homeomorphic to $\mathbb{S}^1$. After sewing on caps, $\riem$ is a domain in the sphere bordered by quasicircles.  
 \begin{example}[all boundaries free, boundary period matrix] \label{ex:boundary_period_genus_zero}
  In this example, we treat all boundaries as free.  Since all boundaries are free, the quantity obtained is invariant under the full modular group, and thus is an invariant in the classical sense. In particular, it is a well-defined function on the Riemann moduli space of multiply-connected domains of fixed connectivity.

  Define the ``boundary period matrix'' as follows. Let $\omega_l$ be the unique harmonic function such that  
  \[ \omega_l(z) = \left\{ \begin{array}{ll} 1 & z \in \partial_l \riem \\ 0 & z \in \partial_m \riem, \ \ l \neq m. \end{array} \right.  \]
  These are sometimes called ``harmonic measures'', e.g. \cite[p 38]{Nehari_book}. 
  The period matrix is 
  \[  \Pi_{lm} = \int_{\partial_l \riem} \ast d\omega_m = \int_{\partial \riem} \omega_l \ast d\omega_m.    \]
  It is a classical fact (proven using Green's identities) that this matrix is symmetric, and if we restrict to $1\leq l,m \leq n-1$, the matrix is positive definite \cite{Nehari_book}. 
  
  Fix now constants $c_1,\ldots,c_n \in \mathbb{R}$ such that $c_1 + \cdots + c_n =0$. We determine the corresponding $\alpha_{\riem}$ as follows. Let $\lambda_l$ be the solution of 
  \[  c_l = \sum_{m=1}^{n-1} \Pi_{l m} \lambda_m    \]
  for $l=1,\ldots,n-1$. Then set 
  \[  \alpha_{\riem} = \sum_{l=1}^{n-1} \lambda_l \partial \omega_l.    \]
  Since 
  \[  \mathrm{Re}(\alpha_{\riem}) =\sum_{l=1}^{n-1} \lambda_l \mathrm{Re}(d\omega)  \]
  and using $\mathrm{Re}( \ast d\omega) = \ast \mathrm{Re}(d\omega)$ we obtain 
  \begin{align*}
   \int_{\partial_l \riem} \ast \mathrm{Re}(\alpha_{\riem}) & = \mathrm{Re} \left( \int_{\partial_l \riem} \sum_{m=1}^{n-1} \lambda_m \ast d\omega_m \right) = \mathrm{Re} \left(  \sum_{m=1}^{n-1}\Pi_{lm} \lambda_m \right) = c_l. 
  \end{align*}
  Furthermore by the fact that $\omega_l$ are constant on each boundary curve, we get that the boundaries are trajectories of $\alpha_{\riem}$.  
  Thus this is the unique $\alpha_{\riem}$ associated to the data $c_1,\ldots,c_n$.  We then compute (observing that $\mathbf{O}\mathrm{Re}(\alpha_{\riem}) =0$) 
  \begin{align*}
   \mathfrak{m}_{\beta,c}(\riem) & = \| \mathrm{Re}(\alpha_{\riem}) \|^2 = \iint_{\riem} dh \wedge \ast dh = \int_{\partial \riem} h \ast dh \\& = 
   \int_{\partial \riem} \left( \sum_{l=1}^{n-1} \lambda_l \omega_l \right) \cdot \left( \sum_{m=1}^{n-1} \lambda_m \ast d\omega_m \right) \\
   & = \sum_{l,m=1}^{n-1} \Pi_{lm} \lambda_l \lambda_m. 
  \end{align*}
  Observe that this is a finite-dimensional family of invariants, one for each choice of $\lambda_1,\ldots,\lambda_{n-1}$. 
 \end{example}
 \begin{example}[no free boundaries, Grunsky inequalities for multiply-connected domains]  \label{ex:Grunsky_genus_zero}
  We let $k=0$. After sewing on caps we obtain that $\mathscr{R}$ is the sphere minus a finite number of punctures $p_1,\ldots,p_n$.  

  Let $\beta_1,\ldots,\beta_n$ be given as in Theorem \ref{th:conformal_invariant_from_field}, where we assume that their residues are all zero.  Let also $\hat{\beta}$ be defined as in \eqref{defn beta}.  Set 
  \[  \overline{\gamma} = (f^{-1})^* \hat{\beta} \in \overline{\mathcal{A}(\Omega)}.     \]  
  By Theorem \ref{th:field_characterization_Faber} we have $\alpha_{\riem} = \mathbf{T}_{\riem} \overline{\gamma}$.

  By \cite[Theorem 3.25]{Schippers_Staubach_scattering_IV}, we have 
  \begin{equation} \label{eq:Grunsky_Bergman_Schiffer_form}
      \| \mathbf{T}_{\riem} \overline{\gamma} \|^2 + \| \mathbf{T}_{\Omega} \overline{\gamma} \|^2 = \| \overline{\gamma} \|^2. 
  \end{equation}
  Furthermore by \cite[Proposition 4.10]{Schippers_Staubach_scattering_III}
  \[  \mathbf{O} \mathbf{T}_{\riem} \overline{\gamma} = - \overline{\gamma} + \mathbf{T}_{\Omega} \overline{\gamma}.   \]
  Putting these together we obtain 
  \begin{align*}
  \mathfrak{m}_{\beta,c}(\riem) & = \| \mathrm{Re} \mathbf{T}_{\riem} \overline{\gamma} \|^2 + \| \mathrm{Re} (-\overline{\gamma} + \mathbf{T}_{\Omega}\overline{\gamma} \|^2 \\ 
  & = \frac{1}{2} \left[ \|  \mathbf{T}_{\riem} \overline{\gamma} \|^2 + \| -\overline{\gamma} + \mathbf{T}_{\Omega} \overline{\gamma} \|^2 \right] \\
  & = \| \overline{\gamma} \|^2 - \mathrm{Re} \left< \gamma,\mathbf{T}_{\Omega} \overline{\gamma} \right>. 
 \end{align*}
 Observe that postivity of $\mathfrak{m}_{\beta,c}$ is an infinite-dimensional family of inequalities, one for each choice of $\overline{\gamma}$, or equivalently $\beta$. These can be written explicitly in terms of the coefficients of the map $f$, where the order of the pole in the principal part of $\beta$ determines the highest order of the Taylor coefficients of $f$ appearing in the expression. It is an advantage of the relative invariants that they can encode higher-order derivatives of the mapping function \cite{Minda_confinv,Sasha_paper}. 

 Positivity of $\mathfrak{m}_{\beta,c}$ is a version of the ``weak Grunsky inequalities'' for multiply-connected domains, while the identity (\ref{eq:Grunsky_Bergman_Schiffer_form}) implies a version of the ``strong Grunsky inequalities'' 
 \[ \| \mathbf{T}_{\Omega} \| \leq 1.  \]
 This operator form of the Grunsky inequalities is goes as far back as Bergman and Schiffer \cite{BergmanSchiffer}, where one can find equation (\ref{eq:Grunsky_Bergman_Schiffer_form}) in the case of one subdomain. These inequalities are often written in terms of the matrix entries derived from coefficients of the Taylor series of the conformal maps. See Duren \cite[Section 4.3]{Durenbook} for the strong and weak forms of the Grunsky inequality in the case of one mapping. The multiply-connected case also has a long history, though it was most often written in terms of coefficients of the Taylor series of the mapping function.
 A full discussion of the history and forms of the Grunsky inequalities would take us far afield; the interested reader is referred to \cite{Durenbook,Pommerenkebook,EMS_survey,Schippers_Staubach_scattering_IV,Shirazi_thesis,Shirazi_Grunsky}. 
 \end{example}

\end{subsection} 
\begin{subsection}{Riemann surfaces of higher genus}
  In the following, let $\riem$ be a surface with $g$ handles and $n$ boundary curves homeomorphic to $\mathbb{S}^1$.  
 \begin{example}[all boundaries free, boundary period matrix] 
   Defining the harmonic measures and period matrix as in Example \ref{ex:boundary_period_genus_zero}, we have  
  \[ \omega_l(z) = \left\{ \begin{array}{ll} 1 & z \in \partial_l \riem \\ 0 & z \in \partial_m \riem, \ \ m \neq l \end{array} \right.  \]
  and 
  \[  \Pi_{lm} = \int_{\partial_l \riem} \ast d\omega_m = \int_{\partial \riem} w_l \ast d\omega_m.    \]

  As in Example \ref{ex:boundary_period_genus_zero}, given $c_1,\ldots,c_n$ satisfying the conditions of Theorem \ref{th:conformal_invariant_from_field} let $\lambda_l$ solve 
   \[  c_l = \sum_{m=1}^{n-1} \Pi_{l m} \lambda_m.    \]
   This is possible because of the well-known fact that $\Pi_{lm}$ is positive definite \cite[Theorem 2.36]{Schippers_Staubach_scattering_I}. 
   We then obtain the conformal invariant 
   \[  \mathfrak{m}_{\beta,c}(\riem) = \sum_{l,m=1}^{n-1} \Pi_{lm} \lambda_l \lambda_m.   \]
 \end{example} 
 \begin{example}[no free boundaries, Grunsky inequalities]
  We extend Example \ref{ex:Grunsky_genus_zero} to general genus; this adds an interaction term to the invariant coming from the handles.
  
  Again let $k=0$. After sewing on caps we obtain that $\mathscr{R}$ is a genus $g$ surface minus a finite number of punctures $p_1,\ldots,p_n$.  

  Let $\beta_1,\ldots,\beta_n$ be given as in Theorem \ref{th:conformal_invariant_from_field}. We assume that their residues are all zero. Set 
  \[  \overline{\gamma} = (f^{-1})^* \hat{\beta} \in \overline{\mathcal{A}(\Omega)}.     \]  
  By Theorem \ref{th:field_characterization_Faber} we have 
  \[  \alpha_{\riem} = \mathbf{T}_{\riem} \overline{\gamma} + \mathbf{S} \gamma.  \]

  The overfare of the one-form $\alpha_{\riem}$ is slightly more delicate in higher genus as it requires a cohomological correction. 
  By \cite[Proposition 7.10]{Schippers_Staubach_scattering_arxiv}
  \[  \mathbf{O} \left( \mathbf{T}_{\riem} \overline{\gamma} + \overline{\mathbf{R}}_{\riem} \overline{\mathbf{S}} \overline{\gamma} \right) = - \overline{\gamma} + \mathbf{T}_{\Omega} \overline{\gamma} + \overline{\mathbf{R}}_{\Omega} \overline{\mathbf{S}} \overline{\gamma}.   \]
  Using the fact that $\mathbf{O}$ is real linear 
  \begin{equation} \label{eq:Grunsky_g_example_temp}
    \mathbf{O} \mathrm{Re} (\alpha_{\riem}) = \mathbf{O} \mathrm{Re} \left( \mathbf{T}_{\riem} \overline{\gamma} + {\mathbf{R}}_{\riem} {\mathbf{S}}  {\gamma} \right) = \mathrm{Re} \left( - \overline{\gamma} + \mathbf{T}_{\Omega} \overline{\gamma} + {\mathbf{R}}_{\Omega} {\mathbf{S}} {\gamma} \right)   
  \end{equation}
  
  We also require a unitarity identity. In the notation here we have the matrix expression 
  \[    \left( \begin{array}{c} \mathbf{T}_{\riem} \overline{\gamma} +  {\mathbf{R}}_{\riem} {\mathbf{S}} {\gamma}  \\ \mathbf{T}_{\Omega} \overline{\gamma} +  {\mathbf{R}}_{\Omega}  {\mathbf{S}}  {\gamma} \\ \overline{\mathbf{S}} \overline{\gamma} \end{array} \right) = \left( \begin{array}{ccc} - \mathbf{T}_\Omega & \mathbf{A} & \mathbf{R}_{\Omega} \\ - \mathbf{T}_{\riem} & \mathbf{B} & \mathbf{R}_{\riem} \\ \overline{\mathbf{S}} & \mathbf{C} & 0  \end{array} \right) \left( \begin{array}{c} - \overline{\gamma} \\ 0 \\  {\mathbf{S}} \gamma \end{array} \right).     \]
  This of course holds for any $\mathbf{A},\mathbf{B},\mathbf{C}$ because of the zero in the right-hand column vector; however, it was shown in \cite[Theorem 3.2.4]{Schippers_Staubach_scattering_IV}, that for certain specific $\mathbf{A},\mathbf{B},\mathbf{C}$ (which play no role here) the matrix is unitary. Thus we have that
  \begin{equation} \label{eq:Grunsky_g_example_temp2}
    \| \mathbf{T}_{\riem} \overline{\gamma} +  {\mathbf{R}}_{\riem} {\mathbf{S}} {\gamma}  \|^2 + \| \mathbf{T}_{\Omega} \overline{\gamma} +  {\mathbf{R}}_{\Omega}  {\mathbf{S}}  {\gamma}\|^2 + \| \overline{\mathbf{S}} \overline{\gamma} \|^2  = \|  \overline{\gamma} \|^2 + \|  {\mathbf{S}}\gamma \|^2.
  \end{equation}
  
  Putting equations (\ref{eq:Grunsky_g_example_temp}) and (\ref{eq:Grunsky_g_example_temp2}) together we obtain 
  \begin{align*}
  \mathfrak{m}_{\beta,c}(\riem) & = \| \mathrm{Re} (\mathbf{T}_{\riem} \overline{\gamma} +  {\mathbf{R}}_{\riem} {\mathbf{S}} {\gamma})  \|^2 + \| \mathrm{Re} (- \overline{\gamma} + \mathbf{T}_{\Omega} \overline{\gamma} + {\mathbf{R}}_{\Omega} {\mathbf{S}} {\gamma}) \|^2 \\ 
  & = \frac{1}{2} \left[ \|  \mathbf{T}_{\riem} \overline{\gamma} +  {\mathbf{R}}_{\riem} {\mathbf{S}} {\gamma}  \|^2 + \|- \overline{\gamma} + \mathbf{T}_{\Omega} \overline{\gamma} + {\mathbf{R}}_{\Omega} {\mathbf{S}} {\gamma} \|^2 \right] \\
  & = \| \overline{\gamma} \|^2 - \mathrm{Re} \left< \gamma,\mathbf{T}_{\Omega} \overline{\gamma} + {\mathbf{R}}_{\Omega} {\mathbf{S}} {\gamma} \right>. 
 \end{align*}
 \end{example}

 This can be thought of as a generalization of the operator form of the Grunsky inequalities to Riemann surfaces of higher genus.  Various versions of the Grunsky inequalities in higher genus were obtained by Shirazi \cite{Shirazi_thesis,Shirazi_Grunsky} and Schippers and Staubach  \cite{Schippers_Staubach_Plemelj,Schippers_Staubach_scattering_IV}. 
\end{subsection}
\end{section}

\end{document}